\documentclass[10pt]{amsart}
\usepackage{amssymb}
\usepackage{bm}
\usepackage{graphicx}
\usepackage[centertags]{amsmath}
\usepackage{amsfonts}
\usepackage{amsthm}
\usepackage{graphicx}
\newtheorem{thm}{Theorem}
\newtheorem{cor}[thm]{Corollary}
\newtheorem{lem}[thm]{Lemma}

\newtheorem{prop}[thm]{Proposition}
\newtheorem{claim}[thm]{Claim}
\newtheorem{fact}[thm]{Fact}

\newtheorem{defn}[thm]{Definition}
\theoremstyle{definition}

\newtheorem{examp}{Example}

\newcommand{\nn}{\mathbb{N}}
\newcommand{\ee}{\varepsilon}

\newcommand{\suc}{\mathrm{Succ}}
\newcommand{\immsuc}{\mathrm{ImmSucc}}
\newcommand{\strong}{\mathrm{Str}}
\newcommand{\ci}{\mathrm{I}}
\newcommand{\mil}{\mathrm{Mil}}
\newcommand{\rel}{\mathrm{Rel}}
\newcommand{\free}{\mathrm{Fr}}

\newcommand{\bfcr}{\mathbf{R}}
\newcommand{\bfcs}{\mathbf{S}}
\newcommand{\bfct}{\mathbf{T}}

\newcommand{\meg}{\geqslant}
\newcommand{\mik}{\leqslant}
\newcommand{\lex}{<_{\mathrm{lex}}}
\newcommand{\con}{\smallfrown}

\begin{document}

\title{Measurable events indexed by trees}

\author{Pandelis Dodos, Vassilis Kanellopoulos and Konstantinos Tyros}

\address{Department of Mathematics, University of Athens, Panepistimiopolis 157 84, Athens, Greece}
\email{pdodos@math.uoa.gr}

\address{National Technical University of Athens, Faculty of Applied Sciences,
Department of Mathematics, Zografou Campus, 157 80, Athens, Greece}
\email{bkanel@math.ntua.gr}

\address{Department of Mathematics, University of Toronto, Toronto, Canada M5S 2E4}
\email{k.tyros@utoronto.ca}

\thanks{2000 \textit{Mathematics Subject Classification}: 05D10, 05C05.}
\thanks{\textit{Key words}: homogeneous trees, strong subtrees, independence.}
\thanks{The first named author was supported by NSF grant DMS-0903558.}

\maketitle


\begin{abstract}
A tree $T$ is said to be homogeneous if it is uniquely rooted and there exists an integer $b\meg 2$, called the branching number
of $T$, such that every $t\in T$ has exactly $b$ immediate successors. We study the behavior of measurable events in probability
spaces indexed by homogeneous trees.

Precisely, we show that for every integer $b\meg 2$ and every integer $n\meg 1$ there exists an integer $q(b,n)$ with the following property.
If $T$ is a homogeneous tree with branching number $b$ and $\{A_t:t\in T\}$ is a family of measurable events in a probability space
$(\Omega,\Sigma,\mu)$ satisfying $\mu(A_t)\meg\ee>0$ for every $t\in T$, then for every $0<\theta<\ee$ there exists a strong subtree
$S$ of $T$ of infinite height such that for every finite subset $F$ of $S$ of cardinality $n\meg 1$ we have
\[ \mu\Big( \bigcap_{t\in F} A_t\Big) \meg \theta^{q(b,n)}. \]
In fact, we can take $q(b,n)= \big((2^b-1)^{2n-1}-1\big)\cdot(2^b-2)^{-1}$. A finite version of this result is also obtained.
\end{abstract}


\section{Introduction}

\subsection{Overview}

Let $(\Omega,\Sigma,\mu)$ be a probability space and $\{A_i:i\in\nn\}$ a family of measurable events in $(\Omega,\Sigma,\mu)$
satisfying $\mu(A_i)\meg \ee>0$ for every $i\in\nn$. It is well-known (and easy to see) that for every $0<\theta<\ee$ there exist
$i,j\in\nn$ with $i\neq j$ such that $\mu(A_i\cap A_j)\meg \theta^2$. Using the classical Ramsey Theorem \cite{Ra} and iterating
this basic fact, we get the following.
\medskip

\noindent \textit{If $\{A_i:i\in\nn\}$ is a family of measurable events in a probability space $(\Omega,\Sigma,\mu)$ satisfying
$\mu(A_i)\meg\ee>0$ for every $i\in\nn$, then for every $0<\theta<\ee$ there exists an infinite subset $L$ of $\nn$ such that
for every integer $n\meg 1$ and every finite subset $F$ of $L$ of cardinality $n$ we have}
\[ \mu \Big(\bigcap_{i\in F} A_i\Big) \meg \theta^n. \]

\noindent In other words, if we are given a sequence of measurable events in a probability space and we are allowed to refine
(i.e. to pass to a subsequence), then we may do as if the events are at least as correlated as if they were independent.

Now suppose that the events are not indexed by the integers but are indexed by another ``structured" set $\mathbb{S}$. A natural
problem is to decide whether the aforementioned result is valid in the new setting. Namely, given a family $\{A_s:s\in\mathbb{S}\}$
of measurable events in a probability space $(\Omega,\Sigma,\mu)$ satisfying $\mu(A_s)\meg\ee>0$ for every $s\in\mathbb{S}$,
is it possible to find a ``substructure" $\mathbb{S}'$ of $\mathbb{S}$ such that the events in the family $\{A_s:s\in\mathbb{S}'\}$
are highly correlated? And if yes, then can we get explicit (and, hopefully, optimal) lower bounds for their joint probability?
Of course what ``substructure" is, will depend on the nature of the index set $\mathbb{S}$.  From a combinatorial perspective,
these questions are of particular importance when the ``structured" set $\mathbb{S}$ is a \textit{Ramsey space}, a notion
introduced by T. J. Carlson in \cite{C} and further developed by S. Todorcevic in \cite{To}.

Various versions have been studied in the literature and several results have been obtained so far. Undoubtedly, the most well-known
and heavily investigated case is when the events are indexed by the Ramsey space $W(\mathbb{A})$ of all finite words over a non-empty
finite alphabet $\mathbb{A}$. Specifically, it was shown by H. Furstenberg and Y. Katznelson in \cite{FK} that for every $0<\ee\mik 1$
and every integer $b\meg 2$ there exists a strictly positive constant $\theta(\ee,b)$ with the following property. If $\mathbb{A}$
is an alphabet with $b$ letters and $\{A_w:w\in W(\mathbb{A})\}$ is a family of measurable events in a probability space $(\Omega,\Sigma,\mu)$ satisfying $\mu(A_w)\meg\ee$ for every $w\in W(\mathbb{A})$, then there exists a combinatorial line $\mathbb{L}$ (see \cite{HJ}) such that
\[ \mu\big( \bigcap_{w\in\mathbb{L}} A_w\big) \meg \theta(\ee,b).\]
In fact, this statement is equivalent to the density Hales--Jewett Theorem. Although powerful, the arguments
in \cite{FK} are not effective and give no estimate on the constant $\theta(\ee,b)$. Explicit lower bounds can be extracted
from the recent ``polymath" proof of the density Hales--Jewett Theorem \cite{Pol}.

Another version has been studied in \cite{DKK}. The events in this case were assumed to be of a rather ``canonical" form and the index set
$\mathbb{S}$ was the level product of a finite sequence of homogeneous trees; we recall that a tree $T$ is said to be \textit{homogeneous}
if it is uniquely rooted and there exists an integer $b\meg 2$, called the \textit{branching number} of $T$, such that every $t\in T$
has exactly $b$ immediate successors. We will not state explicitly this result since this requires a fair amount of terminology.
We point out, however, that it was needed as a tool in the proof of the density version of the Halpern--L\"{a}uchli Theorem \cite{HL}.

\subsection{The main results}

Our goal in this paper is to study the above problem when the index set $\mathbb{S}$ is a (finite or infinite) homogeneous tree
and to obtain explicit and fairly ``civilized" lower bounds. Of course, such a problem can be also studied if the events are indexed
by a boundedly branching tree or, even more generally, by a finitely branching tree. However, as it is shown in Appendix A, the
case of boundedly branching trees is essentially reduced to the case of homogeneous trees, while for finitely branching but
not boundedly branching trees one can construct examples showing that our results do not hold in this wider category.

In the context of trees the most natural (and practically useful) notion of ``substructure" is that of a \textit{strong subtree}.
We recall that a (finite or infinite) subtree $S$ of a uniquely rooted tree $T$ is said to be strong provided that: (a) $S$ is uniquely
rooted and balanced (that is, all maximal chains of $S$ have the same cardinality), (b) every level of $S$ is a subset of some level of $T$,
and (c) for every non-leaf node $s\in S$ and every immediate successor $t$ of $s$ in $T$ there exists a unique immediate successor $s'$ of $s$
in $S$ with $t\mik s'$. The last condition is the most important one and expresses a basic combinatorial requirement, namely that
a strong subtree of $T$ must respect the ``tree structure" of $T$ (it implies, for instance, that a strong subtree of infinite height
of a homogeneous tree is also homogeneous and has the same branching number). Although the notion of a strong subtree was
isolated in the late 1960s, it was highlighted with the work of K. Milliken \cite{Mi1,Mi2} who showed that the family of strong subtrees
of a uniquely rooted and finitely branching tree is partition regular.

\subsubsection{The infinite case}

We are ready to state the first main result of the paper.
\begin{thm} \label{it1}
Let $T$ be a homogeneous tree with branching number $b$. Also let $\{A_t:t\in T\}$ be a family of measurable events in a probability
space $(\Omega,\Sigma,\mu)$ satisfying $\mu(A_t)\meg\ee>0$ for every $t\in T$. Then for every $0<\theta<\ee$ there exists a strong
subtree $S$ of $T$ of infinite height such that for every integer $k\meg 1$ and every strong subtree $R$ of $S$ of height $k$ we have
\begin{equation} \label{ie1}
\mu\Big( \bigcap_{t\in R} A_t\Big) \meg \theta^{p(b,k)}
\end{equation}
where
\begin{equation} \label{ie2}
p(b,k)= \frac{(2^b-1)^k-1}{2^b-2}.
\end{equation}
\end{thm}
Notice that a strong subtree $R$ of height $k$ of a homogeneous tree with branching number $b$ has cardinality
$(b^k-1)/(b-1)$. Therefore, the exponent appearing in the right-hand side of inequality (\ref{ie1}) depends polynomially on the
cardinality of $R$; specifically, if $R$ has cardinality $n$, then the corresponding exponent is $O(n^{b/\log b})$.

It is shown in Appendix B that every non-empty finite subset $F$ of a homogeneous tree is contained in a strong subtree
of height $2|F|-1$. This fact and Theorem \ref{it1} yield the following.
\begin{cor} \label{ic2}
Let $T$ be a homogeneous tree with branching number $b$. Also let $\{A_t:t\in T\}$ be a family of measurable events in a probability
space $(\Omega,\Sigma,\mu)$ satisfying $\mu(A_t)\meg\ee>0$ for every $t\in T$. Then for every $0<\theta<\ee$ there exists a strong
subtree $S$ of $T$ of infinite height such that for every integer $n\meg 1$ and every subset $F$ of $S$ of cardinality $n$ we have
\begin{equation} \label{ie3}
\mu\Big( \bigcap_{t\in F} A_t\Big) \meg \theta^{q(b,n)}
\end{equation}
where
\begin{equation} \label{ie4}
q(b,n)= \frac{(2^b-1)^{2n-1}-1}{2^b-2}.
\end{equation}
\end{cor}

\subsubsection{Free sets: improving the lower bound}

Observe that the integer $q(b,n)$ obtained by Corollary \ref{ic2} depends exponentially on $n$. We do not know whether it is possible
to have polynomial dependence. However, if we restrict our attention to a certain class of finite subsets of homogeneous trees, then
we get optimal lower bounds. This class of finite sets, which we call \textit{free}, is defined in \S 6 in the main text. It includes
various well-known classes of subsets of trees (such as all finite chains, all doubletons and many more) and is sufficiently rich in
the sense that every infinite subset $A$ of a homogeneous tree contains an infinite set $B$ such that every non-empty finite subset
of $B$ is free. Related to this concept, we show the following.
\begin{thm} \label{it3}
Let $T$ be a homogeneous tree. Also let $\{A_t:t\in T\}$ be a family of measurable events in a probability space $(\Omega,\Sigma,\mu)$
satisfying $\mu(A_t)\meg\ee>0$ for every $t\in T$. Then for every $0<\theta<\ee$ there exists a strong subtree $S$ of $T$ of infinite
height such that for every integer $n\meg 1$ and every free subset $F$ of $S$ of cardinality $n$ we have
\begin{equation} \label{ie5}
\mu\Big( \bigcap_{t\in F} A_t\Big) \meg \theta^{n}.
\end{equation}
\end{thm}

\subsubsection{The finite case}

Theorem \ref{it1} has the following finite counterpart which is the third main result of the paper.
\begin{thm} \label{it4}
For every integer $b\meg 2$, every integer $k\meg 1$ and every pair of reals $0<\theta<\ee\mik 1$ there exists an integer
$N$ with the following property. If $T$ is a finite homogenous tree with branching number $b$ and of height at least $N$
and $\{A_t:t\in T\}$ is a family of measurable events in a probability space $(\Omega,\Sigma,\mu)$ satisfying
$\mu(A_t)\meg \ee$ for every $t\in T$, then there exists a strong subtree $S$ of $T$ of height $k$ such that
\begin{equation} \label{ie6}
\mu\Big( \bigcap_{t\in S} A_t\Big) \meg \theta^{p(b,k)}
\end{equation}
where $p(b,k)$ is as in (\ref{ie2}).
\end{thm}
The least integer $N$ with the property described in Theorem \ref{it4} will be denoted by $\mathrm{Cor}(b,k,\theta,\ee)$.
It is interesting to point out that Theorem \ref{it4} does not follow from Theorem \ref{it1} via compactness and one has to
appropriately convert the arguments to the finite setting. An advantage of having an effective proof is that we can extract
explicit and reasonable upper bounds for the integers $\mathrm{Cor}(b,k,\theta,\ee)$; see, for instance, Proposition
\ref{ip5} below.

\subsection{Outline of the proofs}

As we have already mentioned, the proofs of Theorem \ref{it1} and Theorem \ref{it4} are conceptually similar. The main goal
is to construct a strong subtree $W$ of $T$ (which is either infinite, or of sufficiently large height) for which we can control
the joint probability of the events over all \textit{initial} subtrees of $W$. Once this is done, both Theorem \ref{it1} and
Theorem \ref{it4} follow by an application of Milliken's Theorem. The desired strong subtree $W$ is constructed recursively
using the following detailed version of the case ``$k=2$" of Theorem \ref{it4} and which is the basic pigeon-hole principle
in the ``one-step extension" of the recursive selection.
\begin{prop} \label{ip5}
There exists a primitive recursive function $\Phi:\nn^2\to\nn$ such that for every integer $b\meg 2$ and every pair of reals
$0<\theta <\ee\mik 1$ the following holds. If $T$ is a finite homogeneous tree with branching number $b$ such that
\begin{equation} \label{ie7}
h(T)\meg \Phi\Big(b, \Big\lceil \frac{2^b-1}{\ee^{2^b}-\theta^{2^b}}\Big\rceil \Big)
\end{equation}
and $\{A_t:t\in T\}$ is a family of measurable events in a probability space $(\Omega,\Sigma,\mu)$ satisfying $\mu(A_t)\meg \ee$
for every $t\in T$, then there exists a strong subtree $S$ of $T$ of height $2$ such that
\begin{equation} \label{ie8}
\mu\Big( \bigcap_{t\in S} A_t\Big) \meg \theta^{2^b}.
\end{equation}
In particular,
\begin{equation} \label{ie9}
\mathrm{Cor}(b,2,\theta,\ee)\mik \Phi\Big(b, \Big\lceil \frac{2^b-1}{\ee^{2^b}-\theta^{2^b}}\Big\rceil \Big).
\end{equation}
\end{prop}
Proposition \ref{ip5} will be proved in \S 3. The basic ingredient of its proof is an appropriate generalization of the notion of
a ``Shelah line", a fundamental tool in Ramsey Theory introduced by S. Shelah in his work \cite{Sh} on the van der Waerden and the
Hales--Jewett numbers. We call these new combinatorial objects \textit{generalized Shelah lines}.

The proof of Theorem \ref{it3} is somewhat different. In particular, in this case the desired strong subtree $S$ is constructed recursively
and \textit{directly}. The ``one-step extension" of the recursive selection is achieved using the following result.
\begin{prop} \label{ip6}
Let $T$ be a homogeneous tree. Also let $\{A_t:t\in T\}$ be a family of measurable events in a probability space $(\Omega,\Sigma,\mu)$
satisfying $\mu(A_t)\meg\ee>0$ for every $t\in T$. Then for every $0<\theta<\ee$ there exists a strong subtree $S$ of $T$ of infinite
height such that for every $s,t\in S$ we have $\mu (A_s \cap A_t) \meg \theta^2$.
\end{prop}
The main difficulty in the proof of Proposition \ref{ip6} lies in the fact that the class of doubletons of homogeneous trees is \textit{not}
Ramsey; that is, one can find a $2$-coloring of the set of all doubletons of, say, the dyadic tree $D$ such that every strong subtree
of $D$ of height at least $2$ contains doubletons of both colors. These pathologies in Ramsey Theory for trees have been observed in the
late 1960s by F. Galvin and are reflected in his conjecture about partitions of finite subsets of the reals \cite{Ga}, settled in the affirmative
in the early 1980s by A. Blass \cite{Bl}. The key observation in Blass' work is that, for a fixed integer $n\meg 1$, the set of all $n$-element
subsets of certain trees can be categorized in a finite list of classes each of which has the Ramsey property. A similar observation
is also the driving force behind the proof of Proposition \ref{ip6}.

\subsection{Organization of the paper}

The paper is organized as follows. In \S 2 we set up our notation and terminology and we gather some background material needed in the rest
of the paper. In the next section we introduce the aforementioned notion of a generalized Shelah line and we give the proof of Proposition
\ref{ip5}. The proof of Theorem \ref{it4} is given in \S 4 while the proofs of Theorem \ref{it1} and Corollary \ref{ic2} are given in \S 5.
Finally, in \S 6 we define the class of free subsets of homogeneous trees and we give the proofs of Theorem \ref{it3} and Proposition \ref{ip6}.
To facilitate the interested reader we have also included two appendices. In Appendix A we show that Theorem \ref{it1} still holds if the
tree $T$ is merely assumed to be boundedly branching and we provide counterexamples for the case of finitely branching but \textit{not} boundedly
branching trees. In Appendix B we prove that every finite subset $F$ of a homogeneous tree $T$ is contained in a strong subtree of $T$ of height $2|F|-1$, a result needed for the proof of Corollary \ref{ic2}.


\section{Background material}

By $\nn=\{0,1,2,...\}$ we denote the natural numbers. The cardinality of a set $X$ will be denoted by $|X|$.

\subsection{Trees}

By the term \textit{tree} we mean a non-empty partially ordered set $(T,<)$ such that the set $\{s\in T: s<t\}$ is finite and linearly
ordered under $<$ for every $t\in T$. The cardinality of this set is defined to be the \textit{length} of $t$ in $T$ and will be denoted by
$\ell_T(t)$. For every $n\in\nn$ the \textit{$n$-level} of $T$, denoted by $T(n)$, is defined to be the set $\{t\in T: \ell_T(t)=n\}$.
The \textit{height} of $T$, denoted by $h(T)$, is defined as follows. If there exists $k\in\nn$ with $T(k)=\varnothing$, then we set
$h(T)=\max\{n\in\nn: T(n)\neq\varnothing\}+1$; otherwise, we set $h(T)=\infty$.

For every node $t$ of a tree $T$ the set of \textit{successors} of $t$ in $T$ is defined by
\begin{equation} \label{e21}
\suc_T(t)=\{s\in T: t\mik s\}.
\end{equation}
The set of \textit{immediate successors} of $t$ in $T$ is the subset of $\suc_T(t)$ defined by
$\immsuc_T(t)=\{s\in T: t\mik s \text{ and } \ell_T(s)=\ell_T(t)+1\}$.

A \textit{subtree} of a tree $T$ is a subset of $T$ viewed as a tree equipped with the induced partial ordering.
For every $k\in\nn$ with $k<h(T)$ we set
\begin{equation} \label{e22}
T\upharpoonright k= T(0)\cup ... \cup T(k).
\end{equation}
Notice that $h(T\upharpoonright k)=k+1$. An \textit{initial subtree} of $T$ is a subtree of $T$ of the form
$T\upharpoonright k$ for some $k\in\nn$. A \textit{chain} of $T$ is a subset $C$ of $T$ such that for every $s,t\in C$ we have
that either $s\mik t$ or $t\mik s$.

A tree $T$ is said to be \textit{pruned} (respectively, \textit{finitely branching}) if for every $t\in T$ the set of immediate
successors of $t$ in $T$ is non-empty (respectively, finite). It is said to be \textit{boundedly branching} if there exists an
integer $m\meg 1$ such that every $t\in T$ has at most $m$ immediate successors, and it is said to be \textit{balanced} if all
maximal chains of $T$ have the same cardinality. Finally, a tree $T$ is said to be \textit{uniquely rooted} if $|T(0)|=1$;
the \textit{root} of a uniquely rooted tree $T$ is defined to be the node $T(0)$.

Let $T$ be a uniquely rooted tree. For every $s,t\in T$ the \textit{infimum} of $s$ and $t$ in $T$, denoted by $s\wedge_T t$, is defined
to be the $<$-maximal node $w\in T$ such that $w\mik s$ and $w\mik t$ (notice that the infimum is well-defined since $T(0)\mik t$ for every
$t\in T$). More generally, for every non-empty subset $F$ of $T$ the \textit{infimum} of $F$ in $T$, denoted by $\wedge_T F$, is defined to
be the $<$-maximal node $w\in T$ such that $w\mik t$ for every $t\in F$. Observe that $s\wedge_T t=\wedge_T \{s,t\}$.

\subsection{Vector trees}

A \textit{vector tree} $\bfct$ is a non-empty finite sequence of trees having common height; this common height is defined to be the
\textit{height} of $\bfct$ and will be denoted by $h(\bfct)$. We notice that, throughout the paper, we will start the enumeration of
vector trees with $1$ instead of $0$.

The \textit{level product} of a vector tree $\bfct=(T_1,...,T_d)$, denoted by $\otimes\bfct$, is defined to be the set
\begin{equation} \label{e23}
\bigcup_{n< h(\bfct)} T_1(n)\times ...\times T_d(n).
\end{equation}

We say that a vector tree $\bfct=(T_1,...,T_d)$ is \textit{pruned} (respectively, \textit{finitely branching}, \textit{boundedly branching},
\textit{balanced}, \textit{uniquely rooted}) if for every $i\in\{1,...,d\}$ the tree $T_i$ is pruned (respectively, finitely branching,
boundedly branching, balanced, uniquely rooted).

\subsection{Strong subtrees and vector strong subtrees}

A subtree $S$ of a uniquely rooted tree $T$ is said to be \textit{strong} provided that: (a) $S$ is uniquely rooted and balanced, (b) every
level of $S$ is a subset of some level of $T$, and (c) for every non-maximal node $s\in S$ and every $t\in\immsuc_T(s)$ there exists
a unique node $s'\in\immsuc_S(s)$ such that $t\mik s'$. The \textit{level set} of a strong subtree $S$ of $T$ is defined to be the set
\begin{equation} \label{e24}
L_T(S)=\{ m\in\nn: \text{exists } n<h(S) \text{ with } S(n)\subseteq T(m)\}.
\end{equation}
A basic property of strong subtrees is that they preserve infima. That is, if $S$ is a strong subtree of $T$ and $F$ is a non-empty subset
of $S$, then $\wedge_S F=\wedge_T F$.

The concept of a strong subtree is naturally extended to vector trees. Specifically, a \textit{vector strong subtree} of a uniquely rooted
vector tree $\bfct=(T_1,...,T_d)$ is a vector tree $\bfcs=(S_1,...,S_d)$ such that $S_i$ is a strong subtree of $T_i$ for every $i\in\{1,...,d\}$
and $L_{T_1}(S_1)= ...= L_{T_d}(S_d)$.

\subsection{Homogeneous trees and vector homogeneous trees}

Let $b\in\nn$ with $b\meg 2$. By $b^{<\nn}$ we shall denote the set of all finite sequences having values in $\{0,...,b-1\}$.
The empty sequence is denoted by $\varnothing$ and is included in $b^{<\nn}$. We view $b^{<\nn}$ as a tree equipped with the
(strict) partial order $\sqsubset$ of end-extension. Notice that $b^{<\nn}$ is a homogeneous tree with branching number $b$.
For every $n\in\nn$ by $b^n$ we denote the $n$-level of $b^{<\nn}$. If $n\meg 1$, then $b^{<n}$ stands for the initial subtree
of $b^{<\nn}$ of height $n$. By $\lex$ we denote the usual lexicographical order on $b^n$. For every $t,s\in b^{<\nn}$ by
$t^{\con}s$, or simply by $ts$, we shall denote the concatenation of $t$ and $s$.

For technical reasons, that will become transparent below, we will not work with abstract homogeneous trees but with a concrete
subclass. Observe that all homogeneous trees with the same branching number are pairwise isomorphic, and so, such a restriction
will have no effect in the generality of our results.
\medskip

\noindent \textbf{Convention.} \textit{In the rest of the paper by the term ``homogeneous tree" (respectively, ``finite homogeneous tree")
we will always mean a strong subtree of $b^{<\nn}$ of infinite (respectively, finite) height for some integer $b\meg 2$. For every, possibly
finite, homogeneous tree $T$ by $b_T$ we shall denote the branching number of $T$. We follow the same convention for vector trees. In particular,
by the term ``vector homogeneous tree" we will mean a vector strong subtree of $(b_1^{<\nn},...,b_d^{<\nn})$ of infinite height for some
integers $b_1,...,b_d$ with $b_i\meg 2$ for every $i\in\{1,...,d\}$.}
\medskip

\noindent The above convention has two basic advantages. Firstly, it enables us to effectively enumerate the set of immediate successors
of a given node of a, possibly finite, homogeneous tree $T$. Specifically, for every $t\in T$ and every $p\in\{0,...,b_T-1\}$ let
\begin{equation} \label{e25}
t^{\con_T}\!p= \immsuc_T(t)\cap \suc_{b_T^{<\nn}}(t^{\con}p)
\end{equation}
and notice that
\begin{equation} \label{e26}
\immsuc_T(t)=\big\{ t^{\con_T}\!p: p\in\{0,...,b_T-1\}\big\}.
\end{equation}
Also observe that for every $p,q\in\{0,...,b_T-1\}$ we have $t^{\con_T}\!p \lex t^{\con_T}\!q$ if and only if $p<q$.

Secondly, under the above convention, the infimum operation has a particularly simple description. Namely, the infimum of a non-empty
subset $F$ of a, possibly finite, homogeneous tree $T$ is the maximal common initial subsequence of every finite sequence in $F$.
Having this representation in mind, we will drop the subscript in the infinmum operation and we will denote it simply by $\wedge$.

\subsection{Canonical embeddings and canonical isomorphisms}

Let $T$ and $S$ be two, possibly finite, homogeneous trees with the same branching number. We say that a map $f:T\to S$ is a
\textit{canonical embedding} if for every $t,t'\in T$ the following conditions are satisfied.
\begin{enumerate}
\item[(a)] We have $\ell_T(t)=\ell_T(t')$ if and only if $\ell_S\big(f(t)\big)=\ell_S\big(f(t')\big)$.
\item[(b)] We have $t\sqsubset t'$ if and only if $f(t)\sqsubset f(t')$.
\item[(c)] If $\ell_T(t)=\ell_T(t')$, then $t\lex t'$ if and only if $f(t)\lex f(t')$.
\item[(d)] We have $f(t\wedge t')=f(t)\wedge f(t')$.
\end{enumerate}
Observe that a canonical embedding $f:T \to S$ is an injection and its image $f(T)$ is a strong subtree of $S$. Also notice that
if $S$ and $T$ have the same height, then there exists a \textit{unique} bijection between $T$ and $S$ satisfying the above conditions.
This unique bijection will be called the \textit{canonical isomorphism} between $T$ and $S$ and will be denoted by $\ci(T,S)$.

\subsection{Milliken's Theorem}

Let $T$ be a, possibly finite, homogeneous tree. For every integer $k\meg 1$ by $\strong_k(T)$ we shall denote the set
of all strong subtrees of $T$ of height $k$ while by $\mathrm{Str}_{\infty}(T)$ we shall denote the set of all strong subtrees of $T$
of infinite height. For every vector homogeneous tree $\bfct=(T_1,...,T_d)$ the sets $\strong_k(\bfct)$ and $\strong_{\infty}(\bfct)$
are analogously defined. It is easy to see that $\strong_{\infty}(\bfct)$ is a $G_\delta$ (hence Polish) subspace of
$2^{T_1}\times...\times 2^{ T_d}$. We will need the following result due to K. Milliken.
\begin{thm}[\cite{Mi2}] \label{t22}
Let $\bfct$ be a vector homogeneous tree. Then for every Borel subset $\mathcal{C}$ of $\strong_{\infty}(\bfct)$ there exists a vector
strong subtree $\bfcs$ of $\bfct$ of infinite height such that either $\strong_{\infty}(\bfcs)\subseteq \mathcal{C}$ or
$\strong_{\infty}(\bfcs)\cap\mathcal{C}=\varnothing$.

In particular, for every integer $k\meg 1$ and every subset $\mathcal{F}$ of $\strong_k(\bfct)$ there exists a vector strong subtree
$\bfcr$ of $\bfct$ of infinite height such that either $\strong_k(\bfcr)\subseteq \mathcal{F}$ or
$\strong_k(\bfcr)\cap\mathcal{F}=\varnothing$.
\end{thm}
By Theorem \ref{t22} and a standard compactness argument, we get the following.
\begin{cor} \label{tmilfinite}
For every integer $b\meg 2$, every pair of integers $m\meg k\meg 1$ and every integer $r\meg 2$ there exists an integer $M$
with the following property. For every finite homogeneous tree $T$ with branching number $b$ and of height at least $M$
and every $r$-coloring of the set $\strong_k(T)$ there exists a strong subtree $S$ of $T$ of height $m$ such that
the set $\strong_k(S)$ is monochromatic. The least integer $M$ with this property will be denoted by $\mil(b,m,k,r)$.
\end{cor}
Notice that the reduction of Corollary \ref{tmilfinite} to Theorem \ref{t22} via compactness is non-effective and gives no
estimate for the numbers $\mil(b,m,k,r)$. An analysis of the finite version of Milliken's Theorem has been carried out by
M. Soki\'{c} yielding explicit and reasonable upper bounds. In particular, we have the following.
\begin{thm}[\cite{So}] \label{Sokic}
For every integer $k\meg 1$ there exists a primitive recursive function $\phi_k:\nn^3\to \nn$ belonging
to the class $\mathcal{E}^{5+k}$ of Grzegorczyk's hierarchy such that for every integer $b\meg 2$, every integer $m\meg k$
and every integer $r\meg 2$ we have
\begin{equation} \label{e27}
\mil(b,m,k,r) \mik \phi_k(b,m,r).
\end{equation}
\end{thm}

\subsection{Probabilistic preliminaries}

We recall the following well-known fact. The proof is sketched for completeness.
\begin{lem} \label{measure}
Let $0< \theta< \ee\mik 1$ and $N\in\nn$ with $N\meg (\ee^2-\theta^2)^{-1}$. Also let $(A_i)_{i=0}^{N-1}$ be a family of measurable
events in a probability space $(\Omega,\Sigma,\mu)$ satisfying $\mu(A_i)\meg \ee$ for every $i\in \{0,...,N-1\}$. Then there exist
$i,j\in \{0,...,N-1\}$ with $i\neq j$ such that $\mu(A_i\cap A_j) \meg \theta^2$.
\end{lem}
\begin{proof}
For every $i\in\{0,...,N-1\}$ let $\mathbf{1}_{A_i}$ be the indicator function of the event $A_i$ and set
$X=\sum_{i=0}^{N-1} \mathbf{1}_{A_i}$. Then $\mathbb{E}[X]\meg \ee N$ so, by convexity,
\[ \sum_{i\in\{0,...,N-1\}} \sum_{j\in\{0,...,N-1\}\setminus \{i\}} \mu(A_i\cap A_j) =\mathbb{E}[X(X-1)] \meg \ee N (\ee N-1).\]
Therefore, there exist $i,j\in\{0,...,N-1\}$ with $i\neq j$ such that $\mu(A_i\cap A_j)\meg \theta^2$.
\end{proof}
Finally, for every probability space $(\Omega,\Sigma,\mu)$, every $Y\in\Sigma$ with $\mu(Y)>0$ and every $A\in\Sigma$ by
$\mu(A \ | \ Y)$ we shall denote the conditional probability of $A$ relative to $Y$; that is,
\begin{equation} \label{e28}
\mu(A \ | \ Y)=\frac{\mu(A\cap Y)}{\mu(Y)}.
\end{equation}
The conditional probability measure of $\mu$ relative to $Y$ will be denoted by $\mu_Y$. Notice that
$\mu_Y(A)=\mu(A \ | \ Y)$ for every $A\in\Sigma$.


\section{Proof of Proposition \ref{ip5}}

This section is devoted to the proof of Proposition \ref{ip5} stated in the introduction. It is organized as follows. In \S 3.1 we introduce
the class of generalized Shelah lines and we present some of their basic properties. In \S 3.2 we define the primitive recursive function $\Phi$.
The proof of Proposition \ref{ip5} is given in \S 3.3. In \S 3.4 we prove a ``relativized" version of Proposition \ref{ip5}; this ``relativized"
version is needed for the proof of Theorem \ref{it4}. Finally, in \S 3.5 we make some comments concerning the upper bounds for the numbers
$\mathrm{Cor}(b,2,\theta,\ee)$ obtained by Proposition \ref{ip5}.

\subsection{Generalized Shelah lines}

We start with the following definition.
\begin{defn} \label{3d1}
Let $T$ be a finite homogeneous tree of height at least $2$. Also let $F\in\strong_2(T)$ and $P$ be a (possibly empty) subset of
$\{0,...,b_T-1\}$. The $P$-\emph{restriction} of $F$, denoted by $F|_P$, is defined to be the set
\begin{equation} \label{3e1}
F|_ P=\{F(0)\}\cup \big\{F(0)^{\con_F}\!p: p\in P \big\}.
\end{equation}
\end{defn}
Notice that $F|_\varnothing=\{F(0)\}$ and $F|_{\{0,...,b_T-1\}}=F$. Moreover, it easy to see that $F|_{P\cup Q}=F|_P\cup F|_Q$
for any pair $P$ and $Q$ of subsets of $\{0,...,b_T-1\}$. We are ready to introduce the main object of study in this subsection.
\begin{defn}[Standard generalized Shelah lines and their components] \label{3d2}
Let $b\in\nn$ with $b\meg 2$. Also let $i\in\{0,...,b-1\}$, $P\subseteq \{0,...,b-1\}$ and $N\in\nn$ with $N\meg 1$. The
\emph{standard $(b,i,P, N)$-generalized Shelah line}, denoted by $\mathcal{L}(b,i,P,N)$, is the subset of $b^{<N}$ defined by
\begin{equation} \label{3e2}
\mathcal{L}(b,i,P,N)=\bigcup_{k=0}^{N-1} \{i^k\} \cup \big\{i^k p^{N-1-k}: p\in P \big\}.
\end{equation}
For every $k\in\{0,...,N-1\}$ the $k$-\emph{component} of $\mathcal{L}(b,i,P,N)$ is defined by
\begin{equation} \label{3e3}
\mathcal{L}_k(b,i,P,N)=\{i^k\}\cup \big\{i^k p^{N-1-k}: p\in P\big\}.
\end{equation}
\end{defn}
Next we extend Definition \ref{3d2} to all finite homogeneous trees as follows.
\begin{defn}[Generalized Shelah lines of finite homogeneous trees] \label{3d3}
Let $T$ be a finite homogeneous tree and denote by $N$ its height. Also let $i\in\{0,...,b_T-1\}$ and $P\subseteq \{0,...,b_T-1\}$.
The $(i,P)$-\emph{generalized Shelah line} of $T$ is defined to be the image of $\mathcal{L}(b_T,i,P,N)$ under the canonical isomorphism
$\mathrm{I}(b_T^{<N},T)$ between $b_T^{<N}$ and $T$ (see \S 2.5). Respectively, for every $k\in \{0,...,N-1\}$ the $k$-\emph{component}
of the $(i,P)$-generalized Shelah line of $T$ is defined to be the image of the corresponding $k$-component $\mathcal{L}_k(b_T,i,P,N)$
of $\mathcal{L}(b_T,i,P,N)$ under the canonical isomorphism $\mathrm{I}(b_T^{<N},T)$.
\end{defn}
We isolate, below, some basic properties of all generalized Shelah lines of a finite homogeneous tree $T$.
\begin{enumerate}
\item[($\mathcal{P}$1)] Every generalized Shelah line of $T$ is the union of its components.
\item[($\mathcal{P}$2)] The last component of every generalized Shelah line of $T$ is a singleton.
\item[($\mathcal{P}$3)] If $T$ has height $N\meg 2$ and $k\in\{0,...,N-2\}$, then the $k$-component of the $(i,P)$-generalized
Shelah line of $T$ is the $P$-restriction of a strong subtree of $T$ of height $2$.
\end{enumerate}
Properties ($\mathcal{P}$1) and ($\mathcal{P}$2) are straightforward consequences of the relevant definitions. To see property ($\mathcal{P}$3),
consider the $k$-component $\mathcal{L}_k(b_T,i,P,N)$ of the standard generalized Shelah line $\mathcal{L}(b_T,i,P,N)$ and set
\begin{equation} \label{3e4}
F_k=\{i^k\}\cup \big\{ i^k j^{N-1-k}: j\in\{0,...,b_T-1\}\big\}.
\end{equation}
Notice that $F_k\in\strong_2(b_T^{<N})$ and that $F_k|_P=\mathcal{L}_k(b_T,i,P,N)$. Since strong subtrees of height $2$ and their restrictions
are preserved under canonical isomorphisms, we see that property ($\mathcal{P}$3) is also satisfied. The most important property, however, of
generalized Shelah lines is included in the following proposition.
\begin{prop} \label{line}
Let $T$ be a finite homogeneous tree of height at least $2$. Also let $i\in\{0,...,b_T-1\}$ and $P$ be a (possibly empty) subset of
$\{0,...,b_T-1\}$. If $i\notin P$, then the union of any two distinct components of the $(i,P)$-generalized Shelah
line of $T$ contains the $(P\cup\{i\})$-restriction of a strong subtree of $T$ of height $2$.
\end{prop}
\begin{proof}
Clearly we may assume that $T$ is the tree $b_T^{<N}$ where $N$ is the height of $T$. We fix $0\mik k_0<k_1\mik N-1$ and we consider
the following cases.
\medskip

\noindent \textsc{Case 1}: $P=\varnothing$. Let $F\in\strong_2(b_T^{<N})$ be defined by
\[ F=\{i^{k_0}\} \cup \big\{ i^{k_0} j^{k_1-k_0}: j\in\{0,...,b_T-1\}\big\} \]
and observe that $F|_{\{i\}}=\{i^{k_0}\}\cup\{i^{k_1}\}=\mathcal{L}_{k_0}(b_T,i,\varnothing,N)\cup \mathcal{L}_{k_1}(b_T,i,\varnothing,N)$.
\medskip

\noindent \textsc{Case 2}: $P\neq\varnothing$. We set $p=\min P$. Let $G\in\strong_2(b_T^{<N})$ be defined by
\[ G=\{i^{k_0}\} \cup \big\{i^{k_0}j^{N-1-k_0}: j\in\{0,..., b_T-1\} \text{ and } j\neq  i\big\}\cup \{i^{k_1}p^{N-1-k_1}\}.\]
Notice that $G|_P=\mathcal{L}_{k_0}(b_T,i,P,N)$ and $G|_{\{i\}}=\{i^{k_0}\}\cup \{i^{k_1}p^{N-1-k_1}\}$. Thus,
\[ G|_{P\cup\{i\}}= G|_P\cup G|_{\{i\}}\subseteq \mathcal{L}_{k_0}(b_T,i,P,N)\cup\mathcal{L}_{k_1}(b_T,i,P,N). \]
The proof is completed.
\end{proof}

\subsection{The primitive recursive function $\Phi$}

For every $b,i,m\in\nn$ with $b\meg 2$ and $m\meg 2$ we define recursively the integer $M^{(i)}(b,m)$ by the rule
\begin{equation} \label{3e5}
\left\{ \begin{array} {l} M^{(0)}(b,m)=m, \\ M^{(i+1)}(b,m)=\mil\big(b,M^{(i)}(b,m),2,2\big). \end{array} \right.
\end{equation}
Inductively, it is easy to show that
\begin{equation} \label{e32new}
M^{(i)}(b,m)\meg m.
\end{equation}
Moreover, we have the following.
\begin{fact} \label{332fact}
There exists a primitive recursive function $\Phi:\nn^2\to\nn$ belonging to the class $\mathcal{E}^8$ of Grzegorczyk's hierarchy
such that for every integer $b\meg 2$ and every integer $m\meg 2$ we have
\begin{equation} \label{e36}
M^{(b-1)}(b,m) \mik \Phi(b,m).
\end{equation}
\end{fact}
\begin{proof}
The result follows easily by Theorem \ref{Sokic} and elementary properties of primitive recursive functions (see, e.g., \cite{Rose}).
We will provide the details for the benefit of the reader. To this end, we need first to recall some pieces of notation. For every
$j\in\{1,2\}$ by $\pi_j:\nn^2\to\nn$ we denote the projection function to the $j$-coordinate; it belongs to the class $\mathcal{E}^0$.
Also, let $\mathrm{ms}:\nn^2\to\nn$ be the modified substraction function defined by $\mathrm{ms}(n,k)=n-k$ if $n\meg k$ and
$\mathrm{ms}(n,k)=0$ if $n<k$; it belongs to the class $\mathcal{E}^3$.

Now, let $\phi_2:\nn^3\to\nn$ be the primitive recursive function obtained by Theorem \ref{Sokic} for ``$k=2$". Recall that $\phi_2$
belongs to the class $\mathcal{E}^7$ and that for every integer $b\meg 2$, every integer $m\meg 2$ and every integer $r\meg 2$ we have
$\mil(b,m,2,r)\mik \phi_2(b,m,r)$. Define $\psi:\nn^3\to\nn$ by the rule
\[ \left\{ \begin{array} {l} \psi(0,x)=\pi_2(x), \\ \psi(i+1,x)=\phi_2\big( \pi_1(x),\psi(i,x),2\big). \end{array} \right. \]
Since $\phi_2$ belongs to the class $\mathcal{E}^7$, we see that the function $\psi$ belongs to the class $\mathcal{E}^8$.
Finally, let $\Phi:\nn^2\to\nn$ be defined by
\[ \Phi(x)=\psi\big( \mathrm{ms}(\pi_1(x),1),\pi_1(x),\pi_2(x)\big). \]
Clearly the function $\Phi$ belongs to the class $\mathcal{E}^8$. It is easy to check that $\Phi$ is as desired.
\end{proof}

\subsection{Proof of Proposition \ref{ip5}}

By Fact \ref{332fact}, it is enough to show the following.
\begin{lem} \label{3l7}
Let $0<\theta<\ee\mik 1$. Also let $T$ be a finite homogeneous tree such that
\begin{equation} \label{3e7}
h(T)\meg M^{(b_T-1)} \Big(b_T,\Big\lceil\frac{2^{b_T}-1}{\ee^{2^{b_T}}-\theta^{2^{b_T}}}\Big\rceil\Big)
\end{equation}
and $\{A_t: t\in T\}$ be a family of measurable events in a probability space  $(\Omega,\Sigma,\mu)$ satisfying $\mu(A_t)\meg\ee$
for every $t\in T$. Then there exists $F\in\strong_2(T)$ such that
\begin{equation} \label{3e8}
\mu\Big(\bigcap_{t\in F}A_t\Big)\meg \theta^{2^{b_T}}.
\end{equation}
\end{lem}
\begin{proof}
In what follows, for notational simplicity, by $b$ we shall denote the branching number of the tree $T$. We set
\begin{equation} \label{3e9}
\delta=\frac{\ee^{2^b}-\theta^{2^b}}{2^b-1}.
\end{equation}
Recursively, for every $i\in\{0,...,b-1\}$ we will select
\begin{enumerate}
\item[(i)] a positive real $\ee_i$,
\item[(ii)] a positive integer $N_i$ and
\item[(iii)] a strong subtree $R_i$ of $T$
\end{enumerate}
such that the following conditions are satisfied.
\begin{enumerate}
\item[(C1)] We have $\ee_0=\ee$ and $\ee_{i+1}^{2^{i+1}}=\ee_{i}^{2^{i+1}}-\delta$ for every $i\in \{0,...,b-2\}$.
\item[(C2)] For every $i\in \{0,...,b-1\}$ we have $N_{i}=M^{(b-1-i)}\big(b,\lceil \delta^{-1}\rceil\big)$.
\item[(C3)] For every $i\in\{0,...,b-2\}$ the tree $R_{i+1}$ is a strong subtree of $R_i$.
\item[(C4)] For every $i\in\{0,...,b-1\}$ the height of the tree $R_i$ is $N_i$.
\item[(C5)] For every $i\in\{0,...,b-1\}$ and every $F\in\strong_2(R_i)$ we have
\begin{equation} \label{3e10}
\mu\Big(\bigcap_{t\in F|_{\{0,...,i-1\}}} A_t\Big)\meg\ee_i^{2^i}.
\end{equation}
with the convention that $\{0,...,i-1\}=\varnothing$ if $i=0$.
\end{enumerate}
We proceed to the recursive selection. For $i=0$ we set ``$\ee_0=\ee$", ``$N_0=M^{(b-1)}(b,\big\lceil \delta^{-1}\rceil)$" and
``$R_0=T\upharpoonright (N_0-1)$" and we observe that with these choices conditions (C1), (C2) and (C4) are satisfied. Noticing that
$F|_\varnothing=\{F(0)\}$ for every $F\in\strong_2(T)$ we see that condition (C5) is also satisfied. Since condition (C3) is meaningless
in this case, the first step of the recursive selection is completed.

Let $i\in\{0,...,b-2\}$ and assume that the recursive selection has been carried out up to $i$ so that conditions (C1)-(C5) are
satisfied. We start the next step of the recursive selection setting ``$\ee_{i+1}=\big(\ee_i^{2^{i+1}}-\delta\big)^{1/2^{i+1}}$"
and we observe that with this choice condition (C1) is satisfied. Next we set ``$N_{i+1}=M^{(b-1-i-1)}\big(b,\lceil \delta^{-1}\rceil\big)$"
and we notice that condition (C2) is also satisfied. Now, let
\begin{equation} \label{3e11}
\mathcal{F}=\Big\{ F\in \strong_2(R_i) : \mu\Big(\bigcap_{t\in F|_{\{0,...,i\}}} A_t\Big)\meg \ee_{i+1}^{2^{i+1}} \Big\}.
\end{equation}
By our inductive assumptions, the height of the tree $R_i$ is $N_i$. Moreover,
\begin{eqnarray*}
N_{i} & \stackrel{\mathrm{(C2)}}{=} & M^{(b-1-i)}\big(b,\lceil \delta^{-1} \rceil\big) \stackrel{(\ref{3e5})}{=}
\mil\Big( b, M^{(b-1-i-1)} \big( b,\lceil \delta^{-1} \rceil \big),2,2\Big) \\
& = & \mil(b,N_{i+1},2,2).
\end{eqnarray*}
Therefore, by Corollary \ref{tmilfinite}, there exists a strong subtree $R$ of $R_i$ of height $N_{i+1}$
such that either $\strong_2(R)\subseteq \mathcal{F}$ or $\strong_2(R)\cap \mathcal{F}=\varnothing$. We set ``$R_{i+1}=R$"
and we claim that with this choice all the other conditions are satisfied. It is clear that (C3) and (C4) are satisfied, and so,
we only need to check condition (C5). Notice that it is enough to show that $\strong_2(R)\cap\mathcal{F}\neq\varnothing$.
To this end we argue as follows. Let $\mathcal{L}$ be the $(i,\{0,...,i-1\})$-generalized Shelah line of $R$ (recall that,
by convention, we set $\{0,...,i-1\}=\varnothing$ if $i=0$). For every $k\in\{0,...,N_{i+1}-1\}$ let $\mathcal{L}_k$ be the
$k$-component of $\mathcal{L}$ and set
\begin{equation} \label{3e12}
A_k=\bigcap_{t\in \mathcal{L}_k} A_t.
\end{equation}
By property ($\mathcal{P}$3) in \S 3.1, if $k\in \{0,...,N_{i+1}-2\}$, then the $k$-component $\mathcal{L}_k$ of $\mathcal{L}$ is the
$\{0,...,i-1\}$-restriction of some strong subtree of $R$ of height $2$. This fact and condition (C5) of our inductive assumptions
yield that $\mu(A_k)\meg \ee_i^{2^i}$ if $k\in\{0,...,N_{i+1}-2\}$. On the other hand, if $k=N_{i+1}-1$, then by property
($\mathcal{P}$2) in \S 3.1 the $k$-component of $\mathcal{L}$ is a singleton. Noticing that $\ee\meg \ee_{i}^{2^i}$ we conclude that
\begin{equation} \label{3e13}
\mu(A_{k})\meg\ee_{i}^{2^i}
\end{equation}
for every $k\in \{0,...,N_{i+1}-1\}$. Moreover, by the choice of $N_{i+1}$ and $\ee_{i+1}$, we have
\begin{equation} \label{3e14}
N_{i+1} = M^{(b-1-i-1)}\big( b,\lceil \delta^{-1}\rceil\big) \stackrel{(\ref{e32new})}{\meg} \lceil \delta^{-1}\rceil \meg \frac{1}{\delta} =
\frac{1}{(\ee_i^{2^i})^2 -(\ee_{i+1}^{2^i})^2}.
\end{equation}
Hence, by Lemma \ref{measure} applied for ``$N=N_{i+1}$", ``$\ee=\ee_{i}^{2^i}$" and ``$\theta=\ee_{i+1}^{2^i}$",
there exist $0\mik k< k' < N_{i+1}$ such that $\mu(A_k\cap A_{k'})\meg \ee_{i+1}^{2^{i+1}}$. By Proposition \ref{line}, there
exists $G\in\strong_2(R)$ such that $G|_{\{0,...,i-1\}\cup\{i\}}\subseteq \mathcal{L}_{k}\cup\mathcal{L}_{k'}$.
Observing that $G|_{\{0,...,i\}}=G|_{\{0,...,i-1\}\cup\{i\}}$ we see that
\begin{equation} \label{3e15}
\mu\Big(\bigcap_{t\in G|_{\{0,...,i\}}}A_t\Big) \meg \mu\Big(\bigcap_{t\in \mathcal{L}_{k}\cup\mathcal{L}_{k'}}A_t\Big)=
\mu(A_k\cap A_{k'})\meg \ee_{i+1}^{2^{i+1}}.
\end{equation}
Therefore, $G\in\strong_2(R)\cap\mathcal{F}$. This shows that condition (C5) is also satisfied, and so, the recursive selection
is completed.

We isolate, for future use, the following consequence of condition (C1). The proof is left to the interested reader.
\begin{fact} \label{3f8}
For every $i\in\{0,...,b-1\}$ we have $\ee_i^{2^i}\meg \ee^{2^i}- (2^i-1)\delta$.
\end{fact}
We are ready for the final step of the argument. Let $R_{b-1}$ be the strong subtree of $T$ obtained above. We will show
that there exists $F\in\strong_2(R_{b-1})$ satisfying the estimate in (\ref{3e8}). This will finish the proof. To this end we set
\begin{equation} \label{3e17}
r=\ee_{b-1}^{2^{b-1}} \ \text{ and } \ \eta=\big( r^2-\delta\big)^{1/2}.
\end{equation}
By condition (C5) and the choice of $r$, for every $F\in\strong_2(R_{b-1})$ we have
\begin{equation} \label{3e18}
\mu\Big(\bigcap_{t\in F|_{\{0,...,b-2\}}} A_t\Big)\meg r.
\end{equation}
Moreover,
\begin{equation} \label{3e19}
h(R_{b-1}) \stackrel{\mathrm{(C4)}}{=} N_{b-1} \stackrel{\mathrm{(C2)}}{=} M^{(0)}\big( b,\lceil \delta^{-1}\rceil\big)
\stackrel{(\ref{3e5})}{=} \lceil \delta^{-1}\rceil \meg \frac{1}{\delta} \stackrel{(\ref{3e17})}{=}\frac{1}{r^2-\eta^2}.
\end{equation}
Let $\mathcal{G}$ be the $(b-1,\{0,...,b-2\})$-generalized Shelah line of $R_{b-1}$. Also, for every $k\in\{0,...,h(R_{b-1})-1\}$
let $\mathcal{G}_k$ be the $k$-component of $\mathcal{G}$ and set
\begin{equation} \label{3e20}
B_k =\bigcap_{t\in\mathcal{G}_k} A_t.
\end{equation}
Arguing precisely as in the ``one-step extension" of the recursive selection and using the estimates in (\ref{3e18}) and
(\ref{3e19}), it is possible to find $0\mik k< k' < h(R_{b-1})$ and $F\in\strong_2(R_{b-1})$ such that $\mu(B_k\cap B_{k'})\meg \eta^2$
and $F|_{\{0,...,b-2\}\cup\{b-1\}}\subseteq \mathcal{G}_{k}\cup\mathcal{G}_{k'}$. Since $F=F|_{\{0,...,b-2\}\cup\{b-1\}}$ we see that
\begin{equation} \label{3e21}
\mu\Big(\bigcap_{t\in F}A_t\Big)\meg \mu(B_k\cap B_{k'})\meg \eta^2.
\end{equation}
Moreover, by (\ref{3e17}) and Fact \ref{3f8}, we have
\begin{equation} \label{3e22}
\eta^2 \meg \big( \ee^{2^{b-1}}-(2^{b-1}-1)\delta\big)^2 -\delta \meg \ee^{2^b}- (2^b-1)\delta \stackrel{(\ref{3e9})}{=} \theta^{2^b}.
\end{equation}
The proof of Lemma \ref{3l7} is completed.
\end{proof}
As we have already indicated in the beginning of the subsection, having completed the proof of Lemma \ref{3l7}, the proof of Proposition
\ref{ip5} is also completed.

\subsection{Consequences}

We have already mentioned that in this subsection we will give a ``relativized" version of Proposition \ref{ip5}. To this end
we need, first, to introduce some quantitative invariants closely related to the numbers $\mathrm{Cor}(b,2,\theta,\ee)$.
\begin{defn} \label{3d9}
For every integer $b\meg 2$ and every pair of reals $0<\theta<\ee\mik 1$ by $\mathrm{Rel}(b,\theta,\ee)$ we shall
denote the least integer $N$ (if it exists) with the following property. For every finite homogeneous tree $T$ with
branching number $b$ and of height at least $N$ and every family $\{A_t:t\in T\}$ of measurable events in a probability
space $(\Omega,\Sigma,\mu)$ satisfying $\mu(A_t)\meg\ee$ for every $t\in T$, there exists $F\in\strong_2(T)$ such that
\begin{equation} \label{3e23}
\mu\Big(\bigcap_{t\in F(1)}A_t \ \big| \ A_{F(0)}\Big)\meg \theta^{2^b-1}.
\end{equation}
\end{defn}
We will show that the numbers $\mathrm{Rel}(b,\theta,\ee)$ exist. In fact, we shall obtain upper bounds which are expressed in
terms of the Milliken's numbers. Specifically, for every integer $b\meg 2$ and every $0<\theta<\ee\mik 1$ we set
\begin{equation} \label{3e24}
\lambda(b,\theta,\ee)= \big(\ee\cdot \theta^{-1} \big)^{\frac{2^b-1}{2^b+1}}
\end{equation}
and we notice that $\lambda(b,\theta,\ee)>1$. Also let
\begin{equation} \label{3e25}
r(b,\theta,\ee)=\Big\lceil\frac{\ln\ee^{-1}}{\ln\lambda(b,\theta,\ee)}\Big\rceil.
\end{equation}
Finally, for every $i\in\{0,...,r(b,\theta,\ee)\}$ let
\begin{equation} \label{3e26}
\ee_i=\ee\cdot\lambda(b,\theta,\ee)^{i-1}
\end{equation}
and define
\begin{equation} \label{3e27}
m(b,\theta,\ee)=\max\Big\{ M^{(b-1)}\Big(b,\Big\lceil\frac{2^b-1}{\ee_{i}^{2^b}-\ee_{i-1}^{2^b}}\Big\rceil\Big):
1\mik i\mik r(b,\theta,\ee) \Big\}.
\end{equation}
We have the following.
\begin{cor} \label{relcor1prop}
For every integer $b\meg 2$ and every $0<\theta<\ee\mik 1$ we have
\begin{equation} \label{3e28}
\mathrm{Rel}(b,\theta,\ee) \mik \mil\big(b,m(b,\theta,\ee),1,r(b,\theta,\ee)\big).
\end{equation}
\end{cor}
\begin{proof}
The result follows easily by Lemma \ref{3l7} and a stabilization argument. Let us give the details.
For notational simplicity we set $\lambda=\lambda(b,\theta,\ee)$, $r=r(b,\theta,\ee)$ and $m=m(b,\theta,\ee)$.
Also let $\ee_{r+1}=\ee\lambda^r$ and notice that $\ee_{r+1}\meg 1$ by the choice of $r$ in (\ref{3e25}).
Since $\lambda>1$, by (\ref{3e26}), we see that
\[ \ee=\ee_1< \ee_2<...<\ee_{r}< \ee_{r+1}.\]
Let $T$ be a finite homogeneous tree with branching number $b$ and of height at least $\mil(b,m,1,r)$ and $\{A_t:t\in T\}$ be a family
of measurable events in a probability space $(\Omega,\Sigma,\mu)$ satisfying $\mu(A_t)\meg \ee$ for every $t\in T$. There exist a
strong subtree $R$ of $T$ of height $m$ and $i_0\in\{1,...,r\}$ such that for every $t\in R$ we have
\begin{equation} \label{3e29}
\ee_{i_0}\mik \mu\big(A_t\big)\mik\ee_{i_0+1}.
\end{equation}
Therefore, $\mu(A_t)\meg \ee_{i_0}>\ee_{i_0-1}$ and
\[ h(R)=m \stackrel{(\ref{3e27})}{\meg} M^{(b-1)}\Big(b,\Big\lceil\frac{2^b-1}{\ee_{i_0}^{2^b}-\ee_{i_0-1}^{2^b}}\Big\rceil\Big).\]
By Lemma \ref{3l7} applied for ``$\theta=\ee_{i_0-1}$", ``$\ee=\ee_{i_0}$" and the family ``$\{A_t:t\in R\}$",
there exists $F\in\strong_2(R)$ such that
\begin{equation} \label{3e30}
\mu\Big(\bigcap_{t\in F}A_t  \Big)\meg \ee_{i_0-1}^{2^b}.
\end{equation}
By (\ref{3e24}), (\ref{3e26}), (\ref{3e29}) and (\ref{3e30}) and taking into account that $\lambda>1$ and $i_0\meg 1$, we conclude that
\begin{eqnarray*}
\mu\Big(\bigcap_{t\in F(1)}A_t \ \big| \ A_{F(0)}\Big) & = & \frac{\mu\big(\bigcap_{t\in F}A_t\big)}{\mu(A_{F(0)})}
\meg \frac{\ee_{i_0-1}^{2^b}}{\ee_{i_0+1}}= \ee^{2^b-1} \cdot \lambda^{2^b i_0-2^{b+1}-i_0} \\
& \meg & \ee^{2^b-1} \cdot \lambda^{-2^{b}-1} = \theta^{2^b-1}.
\end{eqnarray*}
This shows that $\mathrm{Rel}(b,\theta,\ee) \mik \mil(b,m,1,r)$, as desired.
\end{proof}

\subsection{Comments}

By Fact \ref{332fact} and Lemma \ref{3l7}, the numbers $\mathrm{Cor}(b,2,\theta,\ee)$ are controlled by a primitive recursive function
belonging to the class $\mathcal{E}^8$ of Grzegorczyk's hierarchy. We point out that this upper bound is not optimal and, in fact, we can
have significantly better upper bounds. Precisely, by appropriately modifying the arguments in the proof of Proposition \ref{ip5} (avoiding,
in particular, the use of Milliken's Theorem), it is possible to show the estimate in (\ref{ie9}) is satisfied for the function
$\Psi:\nn^2\to\nn$ defined by
\begin{equation} \label{erem}
\Psi(b,m) = m^{(m+1)^b}.
\end{equation}
Such a modification, however, is technically involved and conceptually less natural to grasp, and so, we prefer to omit it.


\section{Proof of Theorem \ref{it4}}

We fix an integer $b\meg 2$ and a pair of reals $0<\theta<\ee\mik 1$. We will define the numbers $\mathrm{Cor}(b,k,\theta,\ee)$
by recursion on $k$. It is clear that $\mathrm{Cor}(b,1,\theta,\ee)=1$. The definition of the number $\mathrm{Cor}(b,2,\theta,\ee)$
is the content of Proposition \ref{ip5}.

So let $k\in\nn$ with $k\meg 2$ and assume that the number $\mathrm{Cor}(b,k,\theta,\ee)$ has been defined. Let
\begin{equation} \label{4e1}
\eta=\frac{\ee+\theta}{2}
\end{equation}
and set
\begin{equation} \label{4e2}
n(k)= \max\Big\{ \rel(b,\eta,\ee), \mathrm{Cor}(b,k,\theta^{2^b-1},\eta^{2^b-1}) \Big\}+1.
\end{equation}
\begin{claim} \label{4c1}
We have
\begin{equation} \label{4e3}
\mathrm{Cor}(b,k+1,\theta,\ee)\mik \mil(b,n(k),2,2).
\end{equation}
\end{claim}
It is, of course, clear that Theorem \ref{it4} follows by Claim \ref{4c1}. So, what remains is to prove
Claim \ref{4c1}. To this end let $T$ be a finite homogeneous tree with branching number $b$ such that
\begin{equation} \label{4e4}
h(T)\meg \mil(b,n(k),2,2)
\end{equation}
and a family $\{A_t: t\in T\}$ of measurable events in a probability measure space $(\Omega,\Sigma,\mu)$ satisfying
$\mu(A_t)\meg\ee$ for every $t\in T$. We need to find a strong subtree $S$ of $T$ of height $k+1$ such that
\begin{equation} \label{4e5}
\mu\Big( \bigcap_{t\in S} A_t\Big) \meg \theta^{p(b,k+1)}.
\end{equation}

We argue as follows. First we set
\begin{equation} \label{4e6}
\mathcal{F} =\Big\{ F\in\strong_2(T) : \mu\Big( \bigcap_{t\in F(1)} A_t \ \big| \ A_{F(0)} \Big)\meg\eta^{2^b-1}\Big\}.
\end{equation}
By Corollary \ref{tmilfinite} and the estimate in (\ref{4e4}), there exists a strong subtree $R$ of $T$ of height
$n(k)$ such that either $\strong_2(R)\subseteq \mathcal{F}$ or $\strong_2(R)\cap\mathcal{F}=\varnothing$. By the choice of
$n(k)$ made in (\ref{4e2}), we have $n(k)\meg \rel(b,\eta,\ee)$. It follows that $\strong_2(R)\subseteq \mathcal{F}$.

Let $\{r_0<_{\mathrm{lex}}...<_{\mathrm{lex}} r_{b-1}\}$ be the lexicographical increasing enumeration of the $1$-level $R(1)$ of $R$.
Since the height of $R$ is $n(k)$, for every $i\in\{0,...,b-1\}$ we have that $\suc_{R}(r_i)$ is a strong subtree of $R$
of height $n(k)-1$. In particular, $\suc_{R}(r_i)$ is a finite homogeneous tree with branching number $b$ and of height
$n(k)-1$. This observation permits us to consider the canonical isomorphism $\mathrm{I}\big(b^{<n(k)-1},\suc_{R}(r_i)\big)$
between $b^{<n(k)-1}$ and $\suc_{R}(r_i)$. For notational simplicity we shall denote it by $\mathrm{I}_i$.

We set
\begin{equation} \label{e47reallynew}
Y=A_{R(0)}.
\end{equation}
Also, for every $u\in b^{<n(k)-1}$ let
\begin{equation} \label{4e7}
F_u=\{R(0)\} \cup \big\{ \mathrm{I}_i(u): i\in \{0,...,b-1\}\big\}
\end{equation}
and define
\begin{equation} \label{4e8}
B_u=\bigcap_{t\in F_u} A_t\in \Sigma.
\end{equation}
Observe that $F_u\in\strong_2(R)$ with $F_u(1)=\big\{ \mathrm{I}_0(u),...,\mathrm{I}_{b-1}(u)\big\}$ and $F_u(0)=R(0)$.
Since $\strong_2(R)\subseteq\mathcal{F}$, we get that
\begin{equation} \label{4e9}
\mu_Y(B_u) = \frac{\mu(B_u\cap A_{R(0)})}{\mu(A_{R(0)})}= \mu\Big( \bigcap_{t\in F_u(1)}A_t \ \big| \ A_{F_u(0)}\Big)\meg \eta^{2^b-1}.
\end{equation}
Moreover, by (\ref{4e2}), we have $n(k)-1\meg \mathrm{Cor}(b,k,\theta^{2^b-1},\eta^{2^b-1})$. Therefore, applying our inductive
assumptions to the probability space ``$(\Omega,\Sigma,\mu_Y)$" and the family of measurable events ``$\{B_u:u\in b^{<n(k)-1}\}$",
we may find a strong subtree $U$ of $b^{<n(k)-1}$ of height $k$ such that
\begin{equation} \label{4e10}
\mu_Y\Big(\bigcap_{u\in U} B_u\Big) \meg \big(\theta^{2^b-1}\big)^{p(b,k)}.
\end{equation}

We are now in the position to define the desired tree $S$. In particular, let
\begin{equation} \label{4e11}
S=\{R(0)\} \cup \big\{ \mathrm{I}_i(u): u\in U \text{ and } i\in\{0,...,b-1\}\big\}.
\end{equation}
It is easy to see that $S$ is a strong subtree of $T$ of height $k+1$ and with the same root as $R$. Moreover,
\begin{eqnarray} \label{4e12}
\mu\Big(\bigcap_{t\in S}A_t\Big) & \stackrel{(\ref{4e11})}{=} & \mu\Big(A_{R(0)}\cap\bigcap_{u\in U} \bigcap_{i=0}^{b-1} A_{\mathrm{I}_i(u)}\Big)
\stackrel{(\ref{4e7})}{=} \mu\Big( \bigcap_{u\in U} \bigcap_{t\in {F_u}} A_t \Big) \\
& \stackrel{(\ref{4e8})}{=} & \mu\Big(\bigcap_{u\in U} B_u\Big) = \mu\Big( A_{R(0)}\cap \bigcap_{u\in U} B_u\Big) \nonumber \\
& = & \mu(A_{R(0)}) \cdot \mu_Y\Big(\bigcap_{u\in U} B_u\Big) \nonumber  \\
& \stackrel{(\ref{4e10})}{\meg} & \ee\cdot \theta^{(2^b-1)p(b,k)} \meg \theta^{1+(2^b-1)p(b,k)}. \nonumber
\end{eqnarray}
Finally, notice that $p(b,k)=\sum_{i=0}^{k-1}(2^b-1)^i$. Therefore,
\begin{equation} \label{4e13}
1+(2^b-1)p(b,k)=1+\sum_{i=1}^k(2^b-1)^i=\sum_{i=0}^k(2^b-1)^i=p(b,k+1).
\end{equation}
Combining (\ref{4e12}) and (\ref{4e13}), we conclude that the estimate in (\ref{4e5}) is satisfied for the tree $S$.
This completes the proof of Claim \ref{4c1}, and as we have already indicated, the proof of Theorem \ref{it4} is also completed.


\section{Proof of Theorem \ref{it1} and its consequences}

This section is devoted to the proofs of Theorem \ref{it1} and Corollary \ref{ic2} stated in the introduction. We start with the following
lemma which is essentially a multidimensional version of Corollary \ref{relcor1prop}.
\begin{lem} \label{newlemma22}
Let $b\in\nn$ with $b\meg 2$ and $\bfct=(T_1,...,T_d)$ be a vector homogeneous tree such that $b_{T_i}=b$ for every $i\in\{1,...,d\}$.
Also let $\big\{A_t: t\in T_i \text{ and } i\in\{1,...,d\}\big\}$ be a family of measurable events in a probability space
$(\Omega,\Sigma,\mu)$ and $Y\in\Sigma$ with $\mu(Y)>0$ such that for every $(t_1,...,t_d)$ in the level product of
$\bfct$ we have
\begin{equation} \label{e51}
\mu\Big( \bigcap_{i=1}^dA_{t_i} \ \big| \ Y\Big)\meg\ee>0.
\end{equation}
Then for every $0<\theta<\ee$ there exists a vector strong subtree $\bfcs$ of $\bfct$ of infinite height such that for every
$(F_1,...,F_d)\in\strong_2(\bfcs)$ we have
\begin{equation} \label{e52}
\mu\Big( \bigcap_{i=1}^d \bigcap_{t\in F_i(1)} A_t \ \big| \ Y\cap\bigcap_{i=1}^d A_{F_i(0)}\Big)\meg\theta^{2^b-1}.
\end{equation}
\end{lem}
\begin{proof}
We fix $0<\theta<\ee$. Let $\mathcal{F}$ be the set of all vector strong subtrees of $\bfct$ of height $2$ for which the estimate in
(\ref{e52}) is satisfied for the fixed constant $\theta$. By Theorem \ref{t22}, there exists vector strong subtree $\bfcs=(S_1,..., S_d)$
of $\bfct$ of infinite height such that either $\strong_2(\bfcs)\subseteq \mathcal{F}$ or $\strong_2(\bfcs)\cap \mathcal{F}=\varnothing$.
The proof will, of course, be completed once we show that $\strong_2(\bfcs)\cap\mathcal{F}\neq\varnothing$.

To this end we argue as follows. For every $u\in b^{<\nn}$ we set
\begin{equation} \label{e53}
B_u=\bigcap_{i=1}^d A_{\mathrm{I}_i(u)}\cap Y \in \Sigma
\end{equation}
where $\mathrm{I}_i$ stands for the canonical isomorphism $\mathrm{I}(b^{<\nn},S_i)$ between $b^{<\nn}$ and $S_i$ for every
$i\in\{1,...,d\}$. By (\ref{e51}), we have $\mu_Y(B_u)\meg\ee$ for every $u\in b^{<\nn}$. Therefore, by Corollary \ref{relcor1prop},
there exist an integer $N\mik \mathrm{Rel}(b,\theta,\ee)$ and $F\in\strong_2(b^{<N})$ such that
\begin{equation} \label{e54}
\mu_Y \Big(\bigcap_{u\in F(1)}B_u \ \big| \ B_{F(0)}\Big)\meg \theta^{2^b-1}.
\end{equation}
For every $i\in\{1,...,d\}$ we set $F_i=\mathrm{I}_i(F)$. Notice that $(F_1,...,F_d)\in\strong_2(\bfcs)$. Moreover,
\begin{eqnarray*}
\mu\Big( \bigcap_{i=1}^d\bigcap_{t\in F_i(1)}A_t \ \big| \ Y\cap\bigcap_{i=1}^d A_{F_i(0)}\Big) & = &
\frac{\mu\big(\bigcap_{i=1}^d\bigcap_{t\in F_i}A_t\cap Y\big)}{\mu\big(Y\cap\bigcap_{i=1}^dA_{F_i(0)}\big)} =
\frac{\mu\big(\bigcap_{u\in F}B_u\big)}{\mu(B_{F(0)})}\\
& = & \frac{\mu_Y\big(\bigcap_{u\in F}B_u \big)}{\mu_Y( B_{F(0)})} =
\mu_Y\Big( \bigcap_{u\in F(1)}B_u \ \big| \ B_{F(0)}\Big).
\end{eqnarray*}
By (\ref{e54}) and the above equalities, we conclude that $(F_1,...,F_d)\in\strong_2(\bfcs)\cap\mathcal{F}$
and the proof is completed.
\end{proof}
The following lemma is the final step of the proof of Theorem \ref{it1}. It shows that, under the assumptions of Theorem \ref{it1},
we can control the joint probability of the events over all initially subtrees of an appropriately chosen strong subtree of $T$.
\begin{lem} \label{constructing a good subtree}
Let $T$ be a homogeneous tree. Also let $\{A_t:t\in T\}$ be a family of measurable events in a probability space $(\Omega,\Sigma,\mu)$
satisfying $\mu(A_t)\meg\ee>0$ for every $t\in T$. Then for every $0<\theta<\ee$ there exists a strong subtree $W$ of $T$ of infinite
height such that for every $k\in\nn$ we have
\begin{equation} \label{e55}
\mu\Big(\bigcap_{t\in W\upharpoonright k}A_t\Big) \meg \theta^{p(b_T,k+1)}.
\end{equation}
\end{lem}
\begin{proof}
We fix $0<\theta<\ee$. Let us denote by $b$ the branching number of $T$. We set
\begin{equation} \label{e5new1}
\alpha=2^b-1.
\end{equation}
We select a sequence $(\delta_k)$ of reals in the interval $(0,1)$ satisfying
\begin{equation} \label{e56}
\prod_{k\in\nn} (1-\delta_k)\meg\frac{\theta}{\ee}.
\end{equation}
Also let $(\ee_k)$ be the sequence of positive reals defined by the rule
\begin{equation} \label{e57}
\left\{ \begin{array} {l} \ee_0=\ee, \\ \ee_{k+1}=\big(\ee_k(1-\delta_k)\big)^{\alpha}. \end{array} \right.
\end{equation}
We isolate, for future use, the following elementary fact. The proof is left to the interested reader.
\begin{fact} \label{f5}
For every integer $k\meg 1$ we have
\begin{equation} \label{e5new2}
\prod_{i=0}^k \ee_i = \big(\ee^{\sum_{j=0}^k \alpha^j} \big) \cdot
\prod_{i=0}^{k-1} \Big( (1-\delta_i)^{\sum_{j=1}^{k-i} \alpha^j} \Big).
\end{equation}
\end{fact}
Recursively we will select a sequence $(R_k)$ of strong subtrees of $T$ of infinite height such that for every $k\in\nn$
the following conditions are satisfied.
\begin{enumerate}
\item[(C1)] The tree $R_{k+1}$ is a strong subtree of $R_k$.
\item[(C2)] We have $R_{k+1}\upharpoonright k=R_k\upharpoonright k$.
\item[(C3)] We have $\mu\big(\bigcap_{t\in R_k\upharpoonright k} A_t\big)\meg \prod_{i=0}^k \ee_i$.
\item[(C4)] If $\{r_1^k<_{\text{lex}}...<_{\text{lex}} r^k_{b^{k+1}}\}$ is the lexicographical increasing enumeration of
the $(k+1)$-level $R_k(k+1)$ of $R_k$, then for every $(t_1,...,t_{b^{k+1}})$ in the level product of
$\big(\suc_{R_k}(r^k_1),...,\suc_{R_k}(r^k_{b^{k+1}})\big)$ we have
\begin{equation} \label{e57split}
\mu\Big(\bigcap_{i=1}^{b^{k+1}}A_{t_i} \ \big| \ \bigcap_{t\in R_k\upharpoonright k}A_{t}\Big)\meg\ee_{k+1}.
\end{equation}
\end{enumerate}
The recursive selection is somewhat lengthy, and so, we will briefly comment on it for the benefit of the reader.
Conditions (C1) and (C2) are natural and quite common in constructions of this sort. We are mainly interested in condition (C3).
It will be used, later on, to complete the proof of the lemma. Condition (C4) is a technical one. It will be used to show
that the recursive selection can be carried out.

We proceed to the details. For $k=0$ we apply Lemma \ref{newlemma22} for ``$\bfct=(T)$", ``$Y=\Omega$" and
``$\theta=\ee(1-\delta_0)$" and we find a strong subtree $S$ of $T$ of infinite height such that for every $F\in\strong_2(S)$ we have
\begin{equation} \label{e58}
\mu\Big( \bigcap_{t\in F(1)}A_t \ \big| \ A_{F(0)}\Big)\meg\big(\ee(1-\delta_0)\big)^{\alpha}.
\end{equation}
We set ``$R_0=S$" and we observe that with this choice condition (C3) is satisfied. To see that condition (C4) is satisfied,
let $\{r^0_1<_{\text{lex}}...<_{\text{lex}}r^0_b\}$ be the lexicographical increasing enumeration of $R_0(1)$ and fix an
element $(t_1,...,t_{b})$ in the level product of $\big(\suc_{R_0}(r^0_1),...,\suc_{R_0}(r^0_b)\big)$. We set
$F=\{R_0(0)\}\cup\{t_1,...,t_b\}$ and we notice that $F\in\strong_2(R_0)=\strong_2(S)$, $F(0)=R_0(0)$ and $F(1)=\{t_1,...,t_b\}$.
By (\ref{e58}) and taking into account the previous observations and the choice of $\ee_1$ made in (\ref{e57}), we conclude
that condition (C4) is also satisfied. Since conditions (C1) and (C2) are meaningless in this case, the first step of the
recursive selection is completed.

Let $k\in\nn$ and assume that the recursive selection has been carried out up to $k$ so that conditions (C1)-(C4)
are satisfied. Let $\{r^k_1<_{\text{lex}}...<_{\text{lex}} r^k_{b^{k+1}}\}$ be the lexicographical increasing enumeration
of $R_k(k+1)$. Notice that conditions (C3) and (C4) allow us to apply Lemma \ref{newlemma22} for
``$(T_1,...,T_d)=\big(\suc_{R_k}(r^k_1),...,\suc_{R_k}(r^k_{b^{k+1}})\big)$", ``$Y=\bigcap_{t\in R_k\upharpoonright k}A_t$",
``$\ee=\ee_{k+1}$" and ``$\theta=\ee_{k+1}(1-\delta_{k+1})$". Hence, there exists a vector strong subtree $\bfcs=(S_1,..., S_{b^{k+1}})$
of $\big(\suc_{R_k}(r^k_1),...,\suc_{R_k}(r^k_{b^{k+1}})\big)$ of infinite height such that for every
$(F_1,..., F_{b^{k+1}})\in\strong_2(\bfcs)$ we have
\begin{equation} \label{e59}
\mu\Big( \bigcap_{i=1}^{b^{k+1}} \bigcap_{t\in F_i(1)} A_t \ \big| \ \bigcap_{t\in R_k\upharpoonright k}
A_t\cap\bigcap_{i=1}^{b^{k+1}} A_{F_i(0)}\Big) \meg \big(\ee_{k+1}(1-\delta_{k+1})\big)^{\alpha}.
\end{equation}
We set
\begin{equation} \label{e510}
R_{k+1}=(R_k\upharpoonright k) \cup\bigcup_{i=1}^{b^{k+1}} S_i
\end{equation}
and we claim that with this choice conditions (C1)-(C4) are satisfied. Indeed, it is clear that $R_{k+1}$ is a strong subtree of $R_k$
and  $R_{k+1}\upharpoonright k=R_k\upharpoonright k$. Thus, conditions (C1) and (C2) are satisfied. To see that condition (C3)
is satisfied, notice first that
\begin{equation} \label{e511}
R_{k+1}(k+1)=\{ S_1(0)<_{\text{lex}}...<_{\text{lex}} S_{b^{k+1}}(0)\}.
\end{equation}
Since $(S_1,..., S_{b^{k+1}})$ is a vector strong subtree of $\big(\suc_{R_k}(r^k_1),...,\suc_{R_k}(r^k_{b^{k+1}})\big)$,
we see that $\big( S_1(0),..., S_{b^{k+1}}(0)\big)$ is an element of the level product of the vector tree
$\big(\suc_{R_k}(r^k_1),...,\suc_{R_k}(r^k_{b^{k+1}})\big)$. Therefore, by condition (C4) of our inductive assumptions
and (\ref{e511}), we get
\begin{equation} \label{e512}
\mu\Big(\bigcap_{t\in R_{k+1}(k+1)}A_t \ \big| \ \bigcap_{t\in R_k\upharpoonright k}A_{t}\Big)\meg\ee_{k+1}.
\end{equation}
Since $R_{k+1}\upharpoonright k=R_k\upharpoonright k$, we also have that
\begin{equation} \label{e513}
\mu\Big(\bigcap_{t\in R_{k+1}\upharpoonright k+1}A_t\Big) = \mu\Big(\bigcap_{t\in R_k\upharpoonright k}A_t\Big) \cdot
\mu\Big(\bigcap_{t\in R_{k+1}(k+1)} A_t \ \big| \ \bigcap_{t\in R_k\upharpoonright k} A_{t}\Big).
\end{equation}
Therefore, by (\ref{e513}), (\ref{e512}) and condition (C3) of our inductive assumptions,
\begin{equation} \label{e514}
\mu\Big(\bigcap_{t\in R_{k+1}\upharpoonright k+1}A_t \Big)\meg \Big(\prod_{i=0}^{k}\ee_i\Big) \cdot \ee_{k+1} = \prod_{i=0}^{k+1} \ee_i.
\end{equation}
Thus, condition (C3) is satisfied for the tree $R_{k+1}$. So, what remains is to check that condition (C4) is also satisfied.
To this end let $\{r^{k+1}_1<_{\text{lex}}...<_{\text{lex}}r^{k+1}_{b^{k+2}}\}$ be the lexicographical increasing enumeration
of the $(k+2)$-level $R_{k+1}(k+2)$ of the tree $R_{k+1}$. Also let $(t_1,...,t_{b^{k+2}})$ be an arbitrary element of the
level product of $\big(\suc_{R_{k+1}}(r^{k+1}_1),...,\suc_{R_{k+1}}(r^{k+1}_{b^{k+2}})\big)$. For every $i\in\{1,...,b^{k+1}\}$ we define
\begin{equation} \label{e515}
F_i=\{S_i(0)\}\cup\big\{t_{(i-1)b+j}: j\in\{1,...,b\}\big\}.
\end{equation}
Notice that $(F_1,...,F_{b^{k+1}})\in\strong_2(\bfcs)$. Moreover, for every $i\in\{1,...,b^{k+1}\}$,
\begin{equation} \label{e516}
F_i(0)=S_i(0) \ \text{ and } \ F_i(1)=\big\{ t_{(i-1)b+j}: j\in\{1,...,b\} \big\}.
\end{equation}
By (\ref{e516}), we see that
\begin{equation} \label{e517}
\bigcap_{i=1}^{b^{k+1}} \bigcap_{t\in F_i(1)} A_t= \bigcap_{i=1}^{b^{k+2}}A_{t_i}
\end{equation}
while by (\ref{e516}) and (\ref{e511}) and the fact that $R_k\upharpoonright k= R_{k+1}\upharpoonright k$, we have
\begin{equation} \label{e518}
\bigcap_{t\in R_k\upharpoonright k}A_t\cap\bigcap_{i=1}^{b^{k+1}}A_{F_i(0)}= \bigcap_{t\in R_{k+1}\upharpoonright (k+1)}A_t.
\end{equation}
Since $(F_1,...,F_{b^{k+1}})\in\strong_2(\bfcs)$, by (\ref{e59}) and the identities isolated in (\ref{e517}) and (\ref{e518}),
we conclude that
\begin{equation} \label{e519}
\mu\Big(\bigcap_{i=1}^{b^{k+2}} A_{t_i} \ \big| \ \bigcap_{t\in R_{k+1}\upharpoonright (k+1)} A_{t}\Big) \meg
\big(\ee_{k+1}(1-\delta_{k+1})\big)^{\alpha} \stackrel{(\ref{e57})}{=} \ee_{k+2}.
\end{equation}
As $(t_1,...,t_{b^{k+2}})$ was arbitrary, we see that condition (C4) is satisfied. Hence, the recursive selection is completed.

We are now in the position to complete the proof of the lemma. We define
\begin{equation} \label{e520}
W= \bigcup_{k\in\nn} R_k(k).
\end{equation}
By conditions (C1) and (C2), we see that $W$ is a strong subtree of $T$ of infinite height. It suffices to show that
the estimate in (\ref{e55}) holds for every $k\in\nn$. If $k=0$, then this is straightforward. So, let $k\in\nn$ with $k\meg 1$
and observe that $W\upharpoonright k= R_k\upharpoonright k$. Therefore,
\begin{eqnarray*}
\mu\Big(\bigcap_{t\in W\upharpoonright k} A_t\Big) & = & \mu\Big(\bigcap_{t\in R_k \upharpoonright k} A_t\Big)
\stackrel{\mathrm{(C3)}}{\meg} \prod_{i=0}^k \ee_i \\
& \stackrel{(\ref{e5new2})}{=} & \big( \ee^{\sum_{j=0}^k \alpha^j}\big) \cdot \prod_{i=0}^{k-1}
\Big( (1-\delta_i)^{\sum_{j=1}^{k-i} \alpha^j} \Big) \\
& \meg & \big( \ee^{\sum_{j=0}^k \alpha^j} \big) \cdot \Big( \prod_{i=0}^{k-1} (1-\delta_i) \Big)^{\sum_{j=0}^k \alpha^j} \\
& \meg  & \Big( \ee \cdot  \prod_{i\in\nn} (1-\delta_i) \Big)^{\sum_{j=0}^k \alpha^j} \\
& \stackrel{(\ref{e56})}{\meg} & \Big( \ee \cdot \frac{\theta}{\ee}\Big)^{\sum_{j=0}^k \alpha^j}
= \theta^{\sum_{j=0}^k \alpha^j} = \theta^{p(b,k+1)}.
\end{eqnarray*}
The proof of Lemma \ref{constructing a good subtree} is thus completed.
\end{proof}
We are ready to give the proof of Theorem \ref{it1}.
\begin{proof}[Proof of Theorem \ref{it1}]
We fix $0<\theta\mik \ee$. Let
\[ \mathcal{C}=\Big\{ W\in\strong_{\infty}(T): \mu\Big(\bigcap_{t\in W\upharpoonright k} A_t\Big)\meg\theta^{p(b,k+1)}
\text{ for every } k\in\nn \Big\}.\]
It is easy to see that $\mathcal{C}$ is a closed subset of $\strong_{\infty}(T)$. Therefore, by Theorem \ref{t22}
and Lemma \ref{constructing a good subtree}, there exists a strong subtree $S$ of $T$ of infinite height such that
$\strong_{\infty}(S)\subseteq \mathcal{C}$. The strong subtree $S$ is the desired one. Indeed, let $k\in\nn$ with
$k\meg 1$ and $R$ be an arbitrary strong subtree of $S$ of height $k$. There exists a strong subtree $W$ of $S$ of
infinite height such that $R=W\upharpoonright (k-1)$. Since $W\in\strong_{\infty}(S)\subseteq\mathcal{C}$ we see that
\[ \mu\Big(\bigcap_{t\in R} A_t\Big) = \mu\Big(\bigcap_{t\in W\upharpoonright (k-1)} A_t\Big)\meg\theta^{p(b,k)}\]
as desired.
\end{proof}
We proceed to the proof of Corollary \ref{ic2}.
\begin{proof}[Proof of Corollary \ref{ic2}]
Follows by Theorem \ref{it1} and Corollary \ref{AB-cor} in Appendix B.
\end{proof}


\section{Free sets}

This section is organized as follows. In \S 6.1 we introduce the class of free subsets of homogeneous trees and we present some of
their properties. In \S 6.2 we give the proof of Proposition \ref{ip6}. Finally, in \S 6.3 we give the proof of Theorem \ref{it3}.

\subsection {Definition and basic properties}

We start with the following.
\begin{defn} \label{6d1}
Let $T$ be a homogeneous tree. Recursively, for every integer $k\meg 1$ we define a family $\free_k(T)$ of finite subsets of $T$ as follows.
First, let $\free_1(T)$ and $\free_2(T)$ consist of all singletons and all doubletons of $T$ respectively. Let $k\in\nn$ with $k\meg 2$ and
assume that the family $\free_k(T)$ has been defined. Then $\free_{k+1}(T)$ consists of all subsets of $T$ which can be written in the form
$\{t\}\cup G$ where $t\in T$ and $G\in\free_k(T)$ are such that $\ell_T(t)<\ell_T(\wedge G)$. We set
\begin{equation} \label{e61}
\free(T)=\bigcup_{k\meg 1} \free_k(T).
\end{equation}
An element of $\free(T)$ will be called a \emph{free} subset of $T$.
\end{defn}
We have the following characterization of free sets. The proof is straightforward.
\begin{fact} \label{pfree}
Let $T$ be a homogeneous tree and $k\in\nn$ with $k\meg 3$. Also let $F$ be a subset of $T$ of cardinality $k$. Then $F$ is free
if and only if there exists an enumeration $\{t_1,...,t_k\}$ of $F$ such that
\begin{enumerate}
\item[(a)] $\ell_T(t_1)<...<\ell_T(t_{k-1})\mik \ell_T(t_k)$ and
\item[(b)] $\ell_T(t_{m})<\ell_T\big(\wedge\{t_{m+1},...,t_k\}\big)$ for every $m\in \{1,...,k-2\}$.
\end{enumerate}
\end{fact}
Using Fact \ref{pfree} it is easily seen that the class of free sets includes various well-known classes of finite subsets of homogeneous
trees; for instance, all finite chains are free, as well as, the class of ``combs" studied in \cite[\S 6.4]{To}. Moreover, we have the following.
\begin{lem} \label{freelarge}
Every infinite subset $A$ of a homogeneous tree $T$ contains an infinite subset $B$ such that every
non-empty finite subset of $B$ is free.
\end{lem}
\begin{proof}
Recursively, it is possible to select a sequence $(t_n)$ in $A$ such that for every $m\in\nn$ and every non-empty finite subset
$F$ of $\nn$ with $m<\min F$ we have that $\ell_T(t_m)< \ell_T\big( \wedge\{t_n:n\in F\}\big)$. We set $B=\{t_n:n\in\nn\}$.
By Fact \ref{pfree}, we see that every non-empty finite subset of $B$ is free, as desired.
\end{proof}
Finally, we isolate below some elementary properties of all free subsets of a homogeneous tree $T$.
\begin{enumerate}
\item[($\mathcal{P}$1)] If $F\in\free_k(T)$, then $F$ has cardinality $k$.
\item[($\mathcal{P}$2)] If $F\in\free(T)$ and $G$ is a non-empty subset of $F$, then $G\in\free(T)$.
\item[($\mathcal{P}$3)] If $S\in\strong_{\infty}(T)$ and $F\subseteq S$, then $F\in\free(T)$ if and only if $F\in\free(S)$.
\end{enumerate}
Properties ($\mathcal{P}$1) and ($\mathcal{P}$2) are immediate consequences of Definition \ref{6d1}. Property ($\mathcal{P}$3)
follows from the fact that strong subtrees preserve infima.

\subsection{Proof of Proposition \ref{ip6}}

For the proof of Proposition \ref{ip6} we need to do some preparatory work which is of independent interest. To motivate the reader
let us point out that, by Corollary \ref{AB-cor} in Appendix B, every doubleton of a homogeneous tree is contained in a strong subtree
of height $3$. The first step in the proof of Proposition \ref{ip6} is to analyze how this embedding is achieved. As a consequence
of this analysis and Theorem \ref{t22}, the set of all doubletons of a homogeneous tree will be categorized in a finite list of classes
each of which is partition regular. This information will be used, later on, to complete the proof of Proposition \ref{ip6}.

We proceed to the details. In what follows, $T$ will be a homogeneous tree.

\subsection*{Doubletons of type I}

Let $p\in\{0,...,b_T-1\}$ and for every $F\in\strong_3(T)$ we set
\begin{equation} \label{e62-1}
F[p]=\big\{ F(0), F(0)^{\con_F}\!p \big\}.
\end{equation}
We say that a doubleton of $T$ is of \textit{type $\mathrm{I}$ with parameter $(p)$} if it is of the form $F[p]$
for some $F\in\strong_3(T)$. We set
\begin{equation} \label{e62-2}
\mathcal{D}_{(p)}(T)=\big\{ F[p]: F\in\strong_3(T) \big\}.
\end{equation}

\subsection*{Doubletons of type II}

Let $p,q\in\{0,..., b_T-1\}$ with $p\neq q$ and for every $F\in\strong_3(T)$ we set
\begin{equation} \label{e62-3}
F[p,q]=\big\{ F(0)^{\con_F}\!p, F(0)^{\con_F}\!q \big\}.
\end{equation}
We say that a doubleton of $T$ is of \textit{type $\mathrm{II}$ with parameters $(p,q)$} if it is of the form $F[p,q]$
for some $F\in\strong_3(T)$. As above, we set
\begin{equation} \label{e62-4}
\mathcal{D}_{(p,q)}(T)= \big\{ F[p,q]: F\in\strong_3(T)\big\}.
\end{equation}

\subsection*{Doubletons of type III}

Let $p,q,r\in\{0,..., b_T-1\}$ with $p\neq q$. For every $F\in\strong_3(T)$ we set
\begin{equation} \label{e62-5}
F[p,q,r]=\big\{ F(0)^{\con_F}\!p, (F(0)^{\con_F}\!q)^{\con_F}\!r \big\}.
\end{equation}
We say that a doubleton of $T$ is of \textit{type $\mathrm{III}$ with parameters $(p,q,r)$} if it is of the form $F[p,q,r]$
for some $F\in\strong_3(T)$ and we set
\begin{equation} \label{e62-6}
\mathcal{D}_{(p,q,r)}(T)= \big\{ F[p,q,r]: F\in\strong_3(T)\big\}.
\end{equation}
Observe that for every $p\in\{0,...,b_T-1\}$ the class $\mathcal{D}_{(p)}(T)$ is hereditary when passing to strong subtrees;
that is, if $S\in\strong_{\infty}(T)$, then $\mathcal{D}_{(p)}(S)\subseteq \mathcal{D}_{(p)}(T)$. Also notice that, by Theorem
\ref{t22}, for every finite coloring of the set $\mathcal{D}_{(p)}(T)$ there exists $S\in\strong_{\infty}(T)$
such that the set $\mathcal{D}_{(p)}(S)$ is monochromatic. Of course, these properties are also shared by the classes
$\mathcal{D}_{(p,q)}(T)$ and $\mathcal{D}_{(p,q,r)}(T)$. Moreover, we have the following.
\begin{fact} \label{62f1}
Every doubleton of a homogeneous tree $T$ is either of type $\mathrm{I}$, or of type $\mathrm{II}$, or of type $\mathrm{III}$.
\end{fact}
\begin{proof}
Let $s, t\in T$ with $s\neq t$ be arbitrary. We may assume that $\ell_T(s)\mik \ell_T(t)$. We set $w=s\wedge t$ and we consider
the following cases.
\medskip

\noindent \textsc{Case 1}: $\ell_T(w)=\ell_T(s)$. In this case we see that $s=s\wedge t$. Since $s\neq t$, there exists
$p\in\{0,...,b_T-1\}$ such that $t\in \suc_T(s^{\con_T}\!p)$. Therefore, it is possible to select $F\in\strong_3(T)$ such that
$F(0)=s$ and $F(0)^{\con_F}\!p=t$. So in this case the doubleton $\{s,t\}$ is of type $\mathrm{I}$ with parameter $(p)$.
\medskip

\noindent \textsc{Case 2}: \textit{$\ell_T(w)<\ell_T(s)$ and $\ell_T(s)=\ell_T(t)$}. There exist $p,q\in\{0,...,b_T-1\}$ such that
$s\in\suc_T(w^{\con_T}\!p)$ and $t\in\suc_T(w^{\con_T}\!q)$. Observe that $p\neq q$. It is then possible to select $F\in\strong_3(T)$
such that $F(0)=w$, $F(0)^{\con_F}\!p=s$ and $F(0)^{\con_F}\!q=t$. Therefore, in this case the doubleton $\{s,t\}$ is of type
$\mathrm{II}$ with parameters $(p,q)$.
\medskip

\noindent \textsc{Case 3}: $\ell_T(w)<\ell_T(s)<\ell_T(t)$. Notice, first, that there exist $p,q\in\{0,...,b_T-1\}$ with
$p\neq q$ such that $s\in\suc_T(w^{\con_T}\!p)$ and $t\in\suc_T(w^{\con_T}\!q)$. Since $\ell_T(s)<\ell_T(t)$, there exist
$t'\in\suc_T(w^{\con_T}\!q)$ and $r\in\{0,...,b_T-1\}$ such that $\ell_T(t')=\ell_T(s)$ and $t\in\suc_T(t'^{\con_T}\!r)$.
Hence, we may select $F\in\strong_3(T)$ such that $F(0)=w$, $F(0)^{\con_F}\!p=s$, $F(0)^{\con_F}\!q=t'$ and
$(F(0)^{\con_F}\!q)^{\con_F}\!r=t$. It follows that the doubleton $\{s,t\}$ is of type $\mathrm{III}$ with parameters $(p,q,r)$.
The proof is completed.
\end{proof}
We are now ready to proceed to the proof of Proposition \ref{ip6}.
\begin{proof}[Proof of Proposition \ref{ip6}]
We fix $0<\theta<\ee$. Let $p,q,r\in\{0,...,b_T-1\}$ with $p\neq q$ be arbitrary. We set
\begin{equation} \label{e62-7}
\mathcal{F}_{\mathrm{I}}=\Big\{ F\in\strong_3(T): \mu\Big( \bigcap_{t\in F[p]} A_t\Big)\meg \theta^2\Big\},
\end{equation}
\begin{equation} \label{e62-8}
\mathcal{F}_{\mathrm{II}}=\Big\{ F\in\strong_3(T): \mu\Big( \bigcap_{t\in F[p,q]} A_t\Big)\meg \theta^2\Big\}
\end{equation}
and
\begin{equation} \label{e62-9}
\mathcal{F}_{\mathrm{III}}=\Big\{ F\in\strong_3(T): \mu\Big( \bigcap_{t\in F[p,q,r]} A_t\Big)\meg \theta^2\Big\}.
\end{equation}
By Theorem \ref{t22}, there exists $S\in\strong_{\infty}(T)$ such that for every $i\in\{\mathrm{I},\mathrm{II},\mathrm{III}\}$
we have that either $\strong_3(S)\subseteq \mathcal{F}_i$ or $\strong_3(S)\cap\mathcal{F}_i=\varnothing$. Therefore,
by Fact \ref{62f1}, the proof will be completed once we show that $\strong_3(S)\cap\mathcal{F}_i\neq\varnothing$
for every $i\in\{\mathrm{I},\mathrm{II},\mathrm{III}\}$. The argument below is not uniform and depends on the type of
doubletons we are dealing with. We set $N=\lceil(\ee^2-\theta^2)^{-1}\rceil$. Notice that we may (and we will) assume
that $S$ is the tree $b_T^{<\nn}$.
\medskip

\noindent \textsc{Case 1}: \textit{type $\mathrm{I}$ doubletons}. We set $t_k=p^k\in b_T^{<\nn}$ for every $k\in\{0,...,N-1\}$.
By our assumptions, Lemma \ref{measure} can be applied to the family $(A_{t_k})_{k=0}^{N-1}$ and the fixed constant $\theta$.
Hence, there exist $0\mik k_0<k_1<N$ such that $\mu(A_{t_{k_0}}\cap A_{t_{k_1}})\meg \theta^2$. We select $F\in\strong_3(b_T^{<\nn})$
such that
\[ F\upharpoonright 1=\{p^{k_0}\} \cup \big\{ p^{k_0}j^{k_1-k_0}: j\in\{0,...,b_T-1\}\big\}. \]
Since $F[p]=\{p^{k_0},p^{k_1}\}=\{t_{k_0},t_{k_1}\}$, we conclude that $F\in\strong_3(b_T^{<\nn})\cap \mathcal{F}_{\mathrm{I}}$.
\medskip

\noindent \textsc{Case 2}: \textit{type $\mathrm{II}$ doubletons}. In this case we set $s_k=q^k p^{N-1-k}\in b_T^{<\nn}$
for every $k\in\{0,...,N-1\}$. By Lemma \ref{measure}, there exist $0\mik k_0<k_1<N$ such that
$\mu(A_{s_{k_0}}\cap A_{s_{k_1}})\meg \theta^2$. We select $G\in\strong_3(b_T^{<\nn})$ such that
\[ G\upharpoonright 1=\{q^{k_0}\} \cup \big\{ q^{k_0}j^{k_1-k_0}p^{N-1-k_1}: j\in\{0,...,b_T-1\}\big\}. \]
Observe that $G[p,q]=\{q^{k_0}p^{N-1-k_0},q^{k_1}p^{N-1-k_1}\}=\{s_{k_0},s_{k_1}\}$. It follows that
$G\in\strong_3(b_T^{<\nn})\cap \mathcal{F}_{\mathrm{II}}$.
\medskip

\noindent \textsc{Case 3}: \textit{type $\mathrm{III}$ doubletons}. We set $w_k=(qr)^kp\in b_T^{<\nn}$ for every $k\in\{0,...,N-1\}$
where $(qr)^k$ stands for the $k$-times concatenation of $(qr)$ if $k\meg 1$ and $(qr)^0=\varnothing$. Arguing as above, we find
$0\mik k_0<k_1<N$ such that $\mu(A_{w_{k_0}}\cap A_{w_{k_1}})\meg \theta^2$. Let
\begin{eqnarray*}
H & = & \{(qr)^{k_0}\}\cup \big\{(qr)^{k_0}j: j\in\{0,...,b_T-1\}\big\} \cup \\
& & \big\{ (qr)^{k_0}jv(qr)^{k_1-k_0-1}p: j,v\in\{0,...,b_T-1\}\big\}.
\end{eqnarray*}
Notice that $H\in\strong_3(b_T^{<\nn})$ and $H[p,q,r]=\{(qr)^{k_0}p,(qr)^{k_1}p\}=\{w_{k_0},w_{k_1}\}$. Hence,
$H\in\strong_3(b_T^{<\nn})\cap \mathcal{F}_{\mathrm{III}}$. The proof is completed.
\end{proof}
We close this subsection with the following consequence of Proposition \ref{ip6}. It is the analogue
of Corollary \ref{relcor1prop} and it will be used in the proof of Theorem \ref{it3}.
\begin{cor} \label{corfree}
Let $T$ be a homogeneous tree. Also let $\{A_t:t\in T\}$ be a family of measurable events in a probability space $(\Omega,\Sigma,\mu)$
and $Y\in\Sigma$ with $\mu(Y)>0$ such that $\mu(A_t \ | \ Y)\meg\ee>0$ for every $t\in T$. Then for every $0<\theta<\ee$ there exists
$S\in \strong_{\infty}(T)$ such that for every $s,t\in S$ we have $\mu( A_{t} \ | \ Y\cap A_{s})\meg\theta$.
\end{cor}
\begin{proof}
We fix $0<\theta<\ee$ and we set
\begin{equation} \label{e61-1}
\lambda=\big( \ee\cdot \theta^{-1}\big)^{\frac{1}{3}}.
\end{equation}
Notice that $\lambda>1$. Also let
\begin{equation} \label{e61-2}
r=\Big\lceil\frac{\ln\ee^{-1}}{\ln\lambda}\Big\rceil.
\end{equation}
By Theorem \ref{t22} and the choice of $r$, there exist $R\in\strong_{\infty}(T)$ and $i_0\in\{1,...,r\}$ such that for every $t\in R$ we have
\begin{equation} \label{e61-3}
\ee \lambda^{i_0-1}\mik \mu_Y(A_t)\mik \ee \lambda^{i_0}.
\end{equation}
By Proposition \ref{ip6} applied for ``$\theta=\ee \lambda^{i_0-2}$", ``$\ee=\ee \lambda^{i_0-1}$", the family ``$\{A_t:t\in R\}$" and
the probability space ``$(\Omega,\Sigma,\mu_Y)$", there exists $S\in\strong_{\infty}(R)$ such that
\begin{equation} \label{e61-4}
\mu_Y(A_t\cap A_s)\meg \ee^2 \lambda^{2i_0-4}.
\end{equation}
for every $s,t\in S$. By (\ref{e61-3}) and (\ref{e61-4}) and taking into account that $\lambda>1$ and $i_0\meg 1$, we conclude that
\begin{eqnarray*}
\mu(A_t \ | \ Y\cap A_s) & = & \frac{\mu(A_t\cap Y\cap A_s)}{\mu(Y\cap A_s)} = \frac{\mu_Y(A_t\cap A_s)}{\mu_Y( A_s)} \meg
\frac{\ee^2 \lambda^{2i_0-4}}{\ee\lambda^{i_0}} \\
& = & \ee\lambda^{i_0-4} \meg \ee \lambda^{-3} \stackrel{(\ref{e61-1})}{=} \theta
\end{eqnarray*}
for every $s,t\in S$. The proof is completed.
\end{proof}

\subsection{Proof of Theorem \ref{it3}}

Throughout the proof we will use the following notation. For every tree $U$ and every finite subset $F$ of $U$ we set
\begin{equation*}
\mathrm{depth}_U(F)=
\left\{ \begin{array} {ll} \min\{n\in\nn: F\subseteq U\upharpoonright n\} &\text{if $F$ is non-empty}, \\
-1  & \text{otherwise}. \end{array} \right.
\end{equation*}
The quantity $\mathrm{depth}_U(F)$ is called the \textit{depth} of $F$ in $U$ (see, e.g., \cite{To}).

Now, fix $0<\theta<\ee\mik 1$. We select a sequence $(\delta_n)$ in $(0,1)$ such that
\begin{equation} \label{eqfree1}
\prod_{n\in\nn} (1-\delta_n)\meg\frac{\theta}{\ee}.
\end{equation}
Let $(\ee_n)$ be the sequence of positive reals defined recursively by the rule
\begin{equation} \label{eqe}
\left\{ \begin{array} {l} \ee_0=\ee, \\ \ee_{n+1}=\ee_n(1-\delta_n). \end{array} \right.
\end{equation}
Notice that the sequence $(\ee_n)$ is strictly decreasing. Moreover, it is easy to see that for every integer $n\meg 1$ we have
\begin{equation} \label{forme}
\prod_{i=0}^{n}\ee_{i}=\ee^{n+1} \cdot \Big( \prod_{i=0}^{n-1}(1-\delta_i)^{n-i} \Big).
\end{equation}

Recursively, we will select a sequence $(R_n)$ of strong subtrees of $T$ of infinite height such that for every
$n\in\nn$ the following conditions are satisfied.
\begin{enumerate}
\item[(C1)] The tree $R_{n+1}$ is a strong subtree of $R_n$.
\item[(C2)] We have $R_{n+1}\upharpoonright n=R_n\upharpoonright n$.
\item[(C3)] For every finite subset $F$ of $ R_n$ with $\mathrm{depth}_{R_n}(F)\mik n-1$ and every $t\in R_n$
with $n\mik \ell_{R_n}(t)$, if $F\cup\{t\}\in\free(R_n)$ then
\begin{equation} \label{777}
\mu\Big( \bigcap_{w\in F\cup\{t\}} A_w \Big)\meg \prod_{i=0}^{|F|}\ee_i.
\end{equation}
\item[(C4)] For every finite subset $F$ of $R_n$ with $\mathrm{depth}_{R_n}(F)\mik n-1$ and every $s,t\in R_n$
with $s\neq t$ and $n\mik\min\{\ell_{R_n}(s), \ell_{R_n}(t)\}$, if $F\cup\{s, t\}\in\free(R_n)$ then
\begin{equation} \label{eqfreez}
\mu\Big( A_{t} \ \big| \ \bigcap_{w\in F\cup \{s\}} A_w \Big) \meg\ee_{|F|+1}.
\end{equation}
\end{enumerate}
As the reader might have already guess, the above recursive selection is the main step of the proof of Theorem \ref{it3}.
We are mainly interested in conditions (C3) and (C4). The analytical information guaranteed by estimates (\ref{777})
and (\ref{eqfreez}) will be used, later on, to complete the proof of Theorem \ref{it3}.

We proceed to the details. For $n=0$ we apply Corollary \ref{corfree} for ``$Y=\Omega$" and ``$\theta=\ee_1$" and we get a strong
subtree $S$ of $T$ of infinite height such that for every $s,t\in S$ we have $\mu(A_t \ | \ A_s)\meg \ee_1$. We set ``$R_0=S$" and we
observe that with this choice conditions (C3) and (C4) are satisfied. Since (C1) and (C2) are meaningless for $n=0$, the first step
of the recursive selection is completed.

Let $n\in\nn$ and assume that we have selected the trees $R_0,..., R_n$ so that conditions (C1)-(C4) are satisfied.
We need to find the tree $R_{n+1}$. We start with the following fact.
\begin{fact} \label{f63}
Let $F$ be a non-empty finite subset of $R_n$ with $\mathrm{depth}_{R_n}(F)\mik n$ and $t\in R_n$ with $n+1\mik \ell_{R_n}(t)$.
If $F\cup \{t\}\in\free(R_n)$, then the following hold.
\begin{enumerate}
\item[(i)] There exist $k\in\{0,...,n\}$, a (possibly empty) subset $G$ of $R_k$ satisfying $\mathrm{depth}_{R_k}(G)\mik k-1$
and a node $s\in R_k$ with $k=\ell_{R_k}(s)<\ell_{R_k}(t)$ such that $F\cup \{t\}=G\cup\{s,t\}$.
\item[(ii)] We have
\begin{equation} \label{eq512}
\mu\Big(A_{t} \ \big| \ \bigcap_{w\in F} A_w \Big)\meg\ee_{|F|}.
\end{equation}
\end{enumerate}
\end{fact}
\begin{proof}
Part (i) follows by the definition of free sets and conditions (C1) and (C2) of the recursive selection. To see that part (ii)
is also satisfied let $k$, $G$ and $s$ be as in part (i). By property ($\mathcal{P}$3) in \S 6.1 and our inductive assumptions,
we have that $G\cup\{s,t\}\in\free(R_k)$. Therefore, by condition (C4) for the tree $R_k$ applied for the set $G$ and the
doubleton $\{s,t\}$, we see that
\[ \mu\Big(A_{t} \ \big| \ \bigcap_{w\in F}A_w\Big)= \mu\Big( A_{t} \ \big| \ \bigcap_{w\in G\cup \{s\}} A_w \Big) \meg
\ee_{|G|+1}=\ee_{|F|}\]
and the proof is completed.
\end{proof}
The following consequence of Fact \ref{f63} shows that for the selection of the tree $R_{n+1}$ we only have to worry about
conditions (C1), (C2) and (C4).
\begin{cor} \label{cor631}
Let $W\in\strong_{\infty}(R_n)$ be such that $W\upharpoonright n= R_n\upharpoonright n$. Then condition
$\mathrm{(C3)}$ is satisfied if we set $R_{n+1}=W$.
\end{cor}
\begin{proof}
Let $F$ be a finite subset of $W$ satisfying $\mathrm{depth}_{W}(F)\mik n$ and $t\in W$ with $n+1\mik \ell_{W}(t)$ and
assume that $F\cup\{t\}\in\free(W)$. If $F$ is the empty set, then the estimate in (\ref{777}) is straightforward. So we may
assume that $F$ is non-empty. Since $W\in \strong_\infty(R_n)$ and $W\upharpoonright n=R_n\upharpoonright n$, we see that
\begin{enumerate}
\item[(a)] $F$ is a non-empty finite subset of $R_n$ with $\text{depth}_{R_n}(F)\mik n$,
\item[(b)] $n+1\mik \ell_{R_n}(t)$ and
\item[(c)] $F\cup \{t\}\in \free(R_n)$.
\end{enumerate}
By (a), (b) and (c) above and part (ii) of Fact \ref{f63}, we have the estimate
\[\mu\Big(A_{t} \ \big| \ \bigcap_{w\in F} A_w \Big)\meg\ee_{|F|}. \]
Also let $k, G$ and $s$ be as in part (i) of Fact \ref{f63}. Since $R_n\in\strong_{\infty}(R_k)$, by properties ($\mathcal{P}$2)
and ($\mathcal{P}$3) in \S 6.1, we have $G\cup\{s\}\in \free(R_k)$. Hence, by condition (C3) for the tree $R_k$ applied for the
set $G$ and the node $s$, we see that
\[ \mu\Big( \bigcap_{w\in F} A_w\Big)=\mu\Big(\bigcap_{w\in G\cup\{s\}} A_w \Big)\meg \prod_{i=0}^{|G|}\ee_i
=\prod_{i=0}^{|F|-1}\ee_i. \]
Noticing that
\[ \mu\Big(\bigcap_{w\in F\cup \{t\}}A_w\Big) = \mu\Big(A_{t} \ \big| \ \bigcap_{w\in F}A_w\Big) \cdot
\mu\big(\bigcap_{w\in F}A_w\Big) \]
and combining the previous estimates we conclude that condition (C3) is satisfied if we set $R_{n+1}=W$. The proof
of Corollary \ref{cor631} is completed.
\end{proof}
We need one more preparatory step for the selection of the tree $R_{n+1}$.
\begin{claim} \label{631c}
Let $F$ be a finite  subset of $R_n$ such that $\mathrm{depth}_{R_n}(F)\mik n$. Also let $U\in \strong_\infty(R_n)$
with $U\upharpoonright n=R_n\upharpoonright n$. Then there exists $W\in \strong_\infty(U)$ with the following properties.
\begin{enumerate}
\item[(P1)] We have $W\upharpoonright n=U\upharpoonright n$.
\item[(P2)] For every $s,t\in W$ with $s\neq t$ and such that $n+1\mik\min\{\ell_W(s),\ell_W(t)\}$,
if $F\cup\{s,t\}\in\free(W)$ then
\begin{equation} \label{51}
\mu\Big(A_{t} \ \big| \ \bigcap_{w\in F\cup \{s\}}A_w \Big)\meg \ee_{|F|+1}.
\end{equation}
\end{enumerate}
\end{claim}
\begin{proof}
Notice that we may assume that $F$ is non-empty; indeed, for the empty set the result follows by condition (C3) for the tree $R_0$.
Let $\{u_1<_{\mathrm{lex}}...<_{\mathrm{lex}} u_d\}$ be the lexicographical increasing enumeration of the $(n+1)$-level $U(n+1)$ of $U$
(notice that $d=b_T^{n+1}$). Recursively, we will select a family $\big\{ Z_j: j\in\{1,...,d\}\big\}$ of strong subtrees of $T$ such that
the following are satisfied.
\begin{enumerate}
\item[(A1)] For every $j\in\{1,...,d\}$ we have $Z_j\in \strong_{\infty}\big(\suc_U(u_j)\big)$.
\item[(A2)] For every $j\in\{1,...,d-1\}$ we have $L_{T}\big(Z_{j+1}\big)\subseteq L_T\big(Z_{j}\big)$.
\item[(A3)] If $j\in \{1,...,d\}$ is such that $F\cup\{u_j\}\in\free(U)$, then for every  $s,t\in Z_j$
with $s\neq t$ we have
\[\mu\Big(A_{t} \ \big| \ \bigcap_{w\in F\cup\{ s\}}A_w\Big)\meg\ee_{|F|+1}.\]
\end{enumerate}
As the first step is identical to the general one, let us assume that the selection has been carried out up to some
$j\in\{1,...,d-1\}$ so that properties (A1)-(A3) are satisfied. Let $Z$ be a strong subtree of
$\suc_U(u_{j+1})$ such that $L_T(Z)=L_T(Z_{j})$; for the first step we simply set $Z=\suc_U(u_1)$.
We consider the following cases.
\medskip

\noindent \textsc{Case 1}: \textit{the set $F\cup \{u_{j+1}\}$ is not a free subset of $U$}. We set ``$Z_{j+1}=Z$"
and we observe that with this choice properties (A1)-(A3) are satisfied.
\medskip

\noindent \textsc{Case 2}: \textit{the set $F\cup \{u_{j+1}\}$ is a free subset of $U$}. In this case we see that
for every $t\in\suc_U(u_{j+1})$ the set $F\cup\{t\}$ is also a free subset of $U$. Since $U\in\strong_{\infty}(R_n)$,
by part (ii) of Fact \ref{f63}, for every $t\in\suc_U(u_{j+1})$ we have
\[ \mu\Big( A_t \ \big| \ \bigcap_{w\in F}A_w\Big)\meg \ee_{|F|}.\]
We apply Corollary \ref{corfree} for ``$T=Z$", ``$Y=\bigcap_{w\in F}A_w$", ``$\ee=\ee_{|F|}$" and ``$\theta=\ee_{|F|+1}$"
and we get $S\in\strong_{\infty}(Z)$ such that for every $s,t\in S$ we have
\begin{equation} \label{e243}
\mu\Big(A_{t} \ \big| \ \bigcap_{w\in F\cup\{s\}}A_w\Big) = \mu(A_t \ | \ Y\cap A_s)\meg \ee_{|F|+1}.
\end{equation}
We set ``$Z_{j+1}=S$" and we notice that with this choice properties (A1)-(A3) are satisfied. The recursive selection is
thus completed.

Now, for every $j\in\{1,...,d-1\}$ we select a strong subtree $W_j$ of $Z_j$ with $L_T(W_j)=L_T(Z_d)$.
We set $W_d=Z_d$ and we define
\[W=(U\upharpoonright n)\cup\bigcup_{j=1}^{d} W_j.\]
It is clear that $W\in\strong_{\infty}(U)$ and $W\upharpoonright n=U\upharpoonright n$. What remains is to show that
property (P2) is satisfied for the tree $W$. To this end, let $s,t\in W$ with $s\neq t$ and $n+1\mik\min\{\ell_W(s),\ell_W(t)\}$
and assume that $F\cup\{s,t\}\in\free(W)$. Since $W$ is a strong subtree of $U$, we see that
\[\min\{\ell_{U}(s),\ell_{U}(t)\}\meg\min\{\ell_W(s),\ell_W(t)\}\meg n+1.\]
Therefore, by Fact \ref{pfree}, there exists $j_0\in\{1,...,d\}$ such that $F\cup\{u_{j_0}\}\in\free(U)$ and
$s,t\in \suc_U(u_{j_0})\cap W= W_{j_0}\subseteq Z_{j_0}$. Hence, by (A3) above, we conclude that
property (P2) is satisfied. The proof of Claim \ref{631c} is completed.
\end{proof}
After this preliminary discussion we are ready to start the process for selecting the tree $R_{n+1}$. In particular,
let $\{F_1,...,F_m\}$ be an enumeration of the set of all subsets $F$ of $R_n$ with $\mathrm{depth}_{R_n}(F)\mik n$.
By repeated applications of Claim \ref{631c}, it is possible to construct a family $\big\{W_j:j\in\{1,...,m\}\big\}$
of strong subtrees of $R_n$ with the following properties.
\begin{enumerate}
\item[(a)] For every $j\in\{1,...,m\}$ we have $W_j\upharpoonright n=R_n\upharpoonright n$.
\item[(b)] For every $j\in\{1,...,m-1\}$ the tree $W_{j+1}$ is a strong subtree of $W_j$.
\item[(c)] For every $j\in\{1,...,m\}$ and every $s,t\in W_j$ with $s\neq t$ and such that
$n+1\mik\min\{\ell_{W_j}(s),\ell_{W_j}(t)\}$, if $F_j\cup\{s,t\}\in\free(W_j)$ then
\[ \mu\Big(A_{t} \ \big| \ \bigcap_{w\in F_j\cup \{s\}}A_w\Big)\meg \ee_{|F_j|+1}. \]
\end{enumerate}
The construction is fairly standard and the details are left to the reader. We set ``$R_{n+1}=W_m$". By (a), (b)
and (c) above, it is clear that with this choice conditions (C1), (C2) and (C4) are satisfied. On the other hand,
by Corollary \ref{cor631}, condition (C3) is also satisfied. Therefore, the recursive selection is completed.

We are finally in the position to complete the proof of Theorem \ref{it3}. We set
\begin{equation} \label{e63new}
S= \bigcup_{n\in\nn} R_n(n)
\end{equation}
and we observe that $S\in\strong_{\infty}(T)$. We will show that $S$ is the desired strong subtree. So let $G\in\free(S)$
be arbitrary. We need to prove that
\[ \mu\Big( \bigcap_{t\in G} A_t\Big) \meg \theta^{|G|}. \]
To this end, clearly, we may assume that $|G|\meg 2$. We will show, first, that
\begin{equation} \label{eqfreex}
\mu\Big(\bigcap_{t\in G}A_t\Big)\meg\prod_{i=0}^{|G|-1}\ee_i.
\end{equation}
Indeed, by conditions (C1) and (C2) of the recursive selection and the choice of the tree $S$ in (\ref{e63new}),
there exist $n\in\nn$, a (possibly empty) subset $F$ of $R_n$ satisfying $\mathrm{depth}_{R_n}(F)\mik n-1$ and
$s,t\in R_n$ with $s\neq t$ and $n=\ell_{R_n}(s)\mik \ell_{R_n}(t)$ such that $G=F\cup \{s,t\}$.
Since $S\in\strong_{\infty}(R_n)$, by property ($\mathcal{P}$3) in \S 6.1, we see that $F\cup\{s,t\}\in\free(R_n)$.
Therefore, by condition (C4) for the tree $R_n$ applied for the set $F$ and the doubleton $\{s,t\}$, we have
\begin{equation} \label{eqfreexx}
\mu\Big(A_{t} \ \big| \ \bigcap_{w\in F\cup\{s\}} A_w \Big)\meg \ee_{|F|+1}.
\end{equation}
By property ($\mathcal{P}$2) in \S 6.1, we have $F\cup\{s\}\in \free(R_n)$. Thus, by condition (C3),
\begin{equation} \label{eqfreexxx}
\mu\Big(\bigcap_{w\in F\cup\{s\}}A_w\Big) \meg \prod_{i=0}^{|F|}\ee_i.
\end{equation}
Combining (\ref{eqfreexx}) and (\ref{eqfreexxx}), we conclude that the estimate in (\ref{eqfreex}) is satisfied. Therefore,
\begin{eqnarray*}
\mu\Big(\bigcap_{t\in G}A_t\Big) & \stackrel{(\ref{eqfreex})}{\meg} &
\prod_{i=0}^{|G|-1}\ee_i \stackrel{(\ref{forme})}{=} \ee^{|G|} \cdot \Big(\prod_{i=0}^{|G|-2}(1-\delta_i)^{|G|-1-i}\Big) \\
& \meg & \ee^{|G|} \cdot \Big(\prod_{i=0}^{|G|-2}(1-\delta_i)\Big)^{|G|-1} \\
& \meg & \ee^{|G|} \cdot \Big(\prod_{i\in\nn} (1-\delta_i)\Big)^{|G|-1} \\
& \stackrel{(\ref{eqfree1})}{\meg} & \ee^{|G|} \cdot \Big(\frac{\theta}{\ee}\Big)^{|G|-1} \meg
\ee \cdot \theta^{|G|-1} \meg \theta^{|G|}.
\end{eqnarray*}
The proof of Theorem \ref{it3} is thus completed.


\section{Appendix A}

We start by introducing some pieces of notation and terminology. For every finitely branching tree $T$ and every $t\in T$
the \textit{branching number} of $t$ in $T$, denoted by $b_T(t)$, is defined to be cardinality of the set of all immediate
successors of $t$ in $T$. Next we introduce the following class of trees.
\begin{defn} \label{AA-d1}
Let $(b_n)$ be a strictly increasing sequence of positive integers. A tree $T$ will be called $(b_n)$-\emph{large} if it
is uniquely rooted, finitely branching and $b_T(t)\meg b_n$ for every $n\in\nn$ and every $t\in T(n)$.

A tree $T$ will be called \emph{large} if it is $(b_n)$-large for some strictly increasing sequence $(b_n)$
of positive integers.
\end{defn}
We gather, below, some elementary properties of large trees.
\begin{fact} \label{AA-f1}
Let $(b_n)$ be a strictly increasing sequence of positive integers and $T$ be a $(b_n)$-large tree.
Then the following hold.
\begin{enumerate}
\item[(i)] If $S\in\strong_\infty(T)$, then $S$ is $(b_n)$-large.
\item[(ii)] For every strictly increasing sequence $(c_n)$ of positive integers there exists $S\in\strong_\infty(T)$
such that $S$ is $(c_n)$-large.
\end{enumerate}
\end{fact}
We have the following trichotomy.
\begin{prop} \label{AA-p1}
For every uniquely rooted, pruned and finitely branching tree $T$ there exists a strong subtree $S$ of $T$ of infinite height
such that either
\begin{enumerate}
\item[(i)] $S$ is a chain, or
\item[(ii)] $S$ is homogeneous, or
\item[(iii)] $S$ is large.
\end{enumerate}
\end{prop}
\begin{proof}
Assume that neither (i) nor (ii) are satisfied. Recursively and using the Halpern--L\"{a}uchli Theorem \cite{HL}, we may select a sequence $(R_n)$
of strong subtrees of $T$ of infinite height such that for every $n\in\nn$ the following hold.
\begin{enumerate}
\item[(C1)] The tree $R_{n+1}$ is a strong subtree of $R_n$.
\item[(C2)] We have $R_{n+1}\upharpoonright n=R\upharpoonright n$.
\item[(C3)] For every $t\in \bigcup_{m=n}^\infty  R_n(m)$ we have $b_T(t)\meg n+1$.
\end{enumerate}
The recursive selection is fairly standard and the details are left to the reader. Let
\begin{equation} \label{AA-e1}
S=\bigcup_{n\in\nn} R_n(n).
\end{equation}
By conditions (C1) and (C2), we have that $S\in\strong_\infty(T)$. On the other hand, by condition (C3), we see that $S$
is a $(b_n)$-large tree where $b_n=n+1$ for every $n\in\nn$. The proof is completed.
\end{proof}
We remark that, by Proposition \ref{AA-p1}, Theorem \ref{it1} still holds if the events are indexed by a uniquely rooted, pruned
and boundedly branching tree.

On the other hand, if $T$ is a uniquely rooted, pruned and finitely branching tree not containing a strong subtree of infinite height
which is either a chain or homogeneous then, by Proposition \ref{AA-p1} and Fact \ref{AA-f1}, for every strictly increasing sequence $(b_n)$
of positive integers there exists a strong subtree of $T$ which is $(b_n)$-large. Concerning this class of trees we have the following.
\begin{prop} \label{AA-p2}
Let $0<\delta<1$. Also let $(b_n)$ be a strictly increasing sequence of positive integers such that
\begin{equation} \label{sum}
\sum_{n\in\nn} \frac{1}{b_n} \mik \delta.
\end{equation}
Then for every $(b_n)$-large tree $T$ there exists a family $\{A_t:t\in T\}$ of Borel subsets of the interval $[0,1]$ satisfying
$\lambda(A_t)\meg 1-\delta$ for every $t\in T$ and such that
\begin{equation} \label{AA-e2}
\bigcap_{t\in F} A_t=\varnothing
\end{equation}
for every $F\in\strong_2(T)$.
\end{prop}
\begin{proof}
We fix a $(b_n)$-large tree $T$. The family $\{A_t:t\in T\}$ will be defined by recursion on the length of nodes in $T$.
For $n=0$ we set $A_{T(0)}=[0,1]$. Let $n\in\nn$ and assume that we have defined the family $\{A_t:t\in T(n)\}$.
Let $t\in T(n)$ be arbitrary. We partition the set $A_t$ into a family $\{\Delta_s: s\in\immsuc_T(t)\}$ of Borel
sets of equal measure and for every $s\in\immsuc_T(t)$ we set
\begin{equation} \label{AA-e3}
A_s= A_t\setminus \Delta_s.
\end{equation}
We notice two properties guaranteed by the above construction.
\begin{enumerate}
\item[(P1)] For every $t\in T$ and every $w\in\suc_T(t)$ we have $A_w\subseteq A_t$.
\item[(P2)] For every $n\in\nn$, every $t\in T(n)$ and every $s\in\immsuc_T(t)$ we have $\lambda(A_s)= \lambda(A_t) \cdot \big(1
- b_T(t)^{-1}\big)\meg
\lambda(A_t) - 1/b_n$.
\end{enumerate}
Therefore, for every $t\in T$ we have
\[ \lambda(A_t)\meg \lambda(A_{T(0)}) -\sum_{n\in\nn} \frac{1}{b_n} \stackrel{(\ref{sum})}{\meg} 1-\delta.\]
Finally if $F\in\strong_2(T)$, then
\[\bigcap_{t\in F} A_t\subseteq \bigcap_{t\in F(1)} A_t \subseteq \bigcap_{s\in\immsuc_T(F(0))} A_s=\varnothing.\]
The proof is completed.
\end{proof}


\section{Appendix B}

Our goal is this appendix is to give the proof of the following result.
\begin{prop} \label{AB-p}
Let $k\in\nn$ with $k\meg 1$. Then for every uniquely rooted and balanced tree $T$ of height $k$ and every non-empty finite subset
$F$ of $T$ there exists a strong subtree $S$ of $T$ with $h(S)\mik \min\{k,2|F|-1\}$ such that $F\subseteq S$.
\end{prop}
Since every homogeneous tree is uniquely rooted and balanced, by Proposition \ref{AB-p} we get the following.
\begin{cor} \label{AB-cor}
Let $T$ be a homogeneous tree and $n\in\nn$ with $n\meg 1$. Then every subset $F$ of $T$ of cardinality $n$ is contained in a strong
subtree of $T$ of height $2n-1$.
\end{cor}
Before we give the proof of Proposition \ref{AB-p} let us remark that the estimate on the height of the strong subtree obtained by
Corollary \ref{AB-cor} is sharp.
\begin{examp} \label{AB-examp}
For every integer $i\meg 1$ let $t_i=0^{2i}1\in 2^{<\nn}$. Observe that for every pair of integers $1\mik i< j$ we have $t_i\wedge t_j=0^{2i}$.
Now, fix an integer $n\meg 2$ and set $A_n=\big\{t_i: i\in\{1,...,n\}\big\}$. Let $S$ be an arbitrary strong subtree of $2^{<\nn}$ with
$A_n\subseteq S$. Notice, first, that the level set of $S$ must contain the set $\big\{2i+1:i\in\{1,...,n\}\big\}$. Since strong subtrees
preserve infima, we see that $\{t_i\wedge t_{i+1}:i\in\{1,...,n-1\}\big\}\subseteq S$, and so, the level set of $S$ must also contain the
set $\big\{2i:i\in\{1,...,n-1\}\big\}$. Therefore, the height of $S$ is at least $2n-1$.
\end{examp}
We proceed to the proof of Proposition \ref{AB-p}.
\begin{proof}[Proof of Proposition \ref{AB-p}]
The result will be proved by induction on $k$. The case $k=1$ is straightforward. Let $k\in\nn$ with $k\meg 1$ and assume that the
result has been proved for every uniquely rooted and balanced tree of height at most $k$. Let $T$ be a uniquely rooted and balanced
tree of height $k+1$ and $F$ be a non-empty finite subset of $T$. We need to find a strong subtree $S$ of $T$ with
$h(S)\mik \min\{k+1,2|F|-1\}$ such that $F\subseteq S$. Clearly we may assume that $|F|\meg 2$.

Let $w_0=\wedge_T F$ be the infimum of $F$ in $T$ and set
 \[I(F)=\big\{ t\in\immsuc_{T}(w_0):  F\cap \suc_{T}(t)\neq \varnothing\big\}.\]
Notice that
\begin{equation} \label{eqh}
\bigcup_{t\in I(F)} \big(F\cap \suc_{T}(t)\big) \subseteq F\subseteq \{w_0\}\cup \bigcup_{t\in I(F)} \big(F\cap \suc_{T}(t)\big)
\end{equation}
and so
\begin{equation} \label{eqhhh}
\sum_{t\in I(F)} |F\cap \suc_{T}(t)| \mik |F| \mik 1 + \sum_{t\in I(F)} |F\cap \suc_{T}(t)|.
\end{equation}
Observe that $I(F)$ is non-empty (for if not, by (\ref{eqh}), we would have that $F=\{w_0\}$).

Let $t\in I(F)$ be arbitrary. Since $T$ is a balanced tree of height $k+1$, we see that $\suc_T(t)$ is a uniquely rooted and balanced tree
of height at most $k$. By our inductive assumptions, there exists a strong subtree $W_t$ of $\suc_{T}(t)$ such that
\begin{equation} \label{eqx}
F\cap\suc_{T}(t)\subseteq W_t
\end{equation}
and
\begin{equation} \label{eqnew}
h(W_t)\mik 2|F\cap\suc_{T}(t)|-1.
\end{equation}
Observe that $|F\cap \suc_{T}(t)|<|F|$ (for if not, we would have that $F\subseteq \suc_{T}(t)$ which yields that
$w_0\in\suc_T(t)$, a contradiction). Therefore,
\begin{equation} \label{hw}
h(W_t)\mik 2|F|-2.
\end{equation}

We set
\begin{equation} \label{L}
L=\bigcup_{t\in I(F)}L_T(W_t)
\end{equation}
and we select a family $\{S_t:t\in\immsuc_T(w_0)\}$ of strong subtrees of $T$ such that
\begin{enumerate}
\item[(P1)] $S_t\subseteq \suc_T(t)$ and $L_T(S_t)=L$ for every $t\in\immsuc_{T}(w_0)$, and
\item[(P2)] $W_t\subseteq S_t$ for every $t\in I(F)$.
\end{enumerate}
Such a selection is possible since the tree $T$ is balanced. Finally, let
\begin{equation} \label{s}
S=\{w_0\} \cup \big\{S_t: t\in\immsuc_{T}(w_0)\big\}.
\end{equation}
By (\ref{eqh}) and properties (P1) and (P2), we see that $S$ is a strong subtree of $T$ and that $F\subseteq S$.
The proof will be completed once we show that $h(S)\mik 2|F|-1$. Indeed notice that, by (\ref{s}) and property (P1), we have
\begin{equation}
\label{hs}h(S)=|L|+1.
\end{equation}
We consider the following cases.
\medskip

\noindent \textsc{Case 1}: $|I(F)|=1$. Let $t_0\in\immsuc_{T}(w_0)$ be the unique element of $I(F)$. By (\ref{L}),
we have $L=L_T(W_{t_0})$. Hence,
\[ h(S)\stackrel{(\ref{hs})}{=}|L|+1=|L_T(W_{t_0})|+1=h(W_{t_0})+1\stackrel{(\ref{hw})}{\mik} 2|F|-1. \]
\noindent \textsc{Case 2:} $|I(F)|\meg 2$. Notice that
\begin{eqnarray*}
|L| & \stackrel{(\ref{L})}{\mik} & \sum_{t\in I(F)} |L_{T}(W_t)| \stackrel{(\ref{eqnew})}{\mik}
2 \sum_{t\in I(F)} |F\cap \suc_{T}(t)|- |I(F)| \\
& \stackrel{(\ref{eqhhh})}{\mik} & 2|F|-|I(F)|\mik 2|F|-2.
\end{eqnarray*}
Combining (\ref{hs}) and the above estimate we conclude that $h(S)\mik 2|F|-1$.
The above case are exhaustive, and so, the proof is completed.
\end{proof}



\begin{thebibliography}{99}

\bibitem{Bl} A. Blass, \textit{A partition theorem for perfect sets}, Proc. Amer. Math. Soc., 82 (1981), 271-277.

\bibitem{DKK} P. Dodos, V. Kanellopoulos and N. Karagiannis, \textit{A density version of the Halpern--L\"{a}uchli theorem},
preprint (2010), available at \verb"http://arxiv.org/abs/1006.2671".

\bibitem{C} T. J. Carlson, \textit{Some unifying principles in Ramsey Theory}, Discrete Math., 68 (1988), 117-169.

\bibitem{FK} H. Furstenberg and Y. Katznelson, \textit{A density version of the Hales--Jewett theorem}, Journal d'Anal. Math., 57 (1991), 64-119.

\bibitem{Ga} F. Galvin, \textit{Partition theorems for the real line}, Notices Amer. Math. Soc., 15 (1968), 660.

\bibitem{HJ} A. H. Hales and R. I. Jewett, \textit{Regularity and positional games}, Trans. Amer. Math. Soc., 106 (1963), 222-229.

\bibitem{HL} J. D. Halpern and H. L\"{a}uchli, \textit{A partition theorem}, Trans. Amer. Math. Soc., 124 (1966), 360-367.

\bibitem{Mi1} K. Milliken, \textit{A Ramsey theorem for trees}, J. Comb. Theory Ser. A, 26 (1979), 215-237.

\bibitem{Mi2} K. Milliken, \textit{A partition theorem for the infinite subtrees of a tree}, Trans. Amer. Math. Soc., 263 (1981), 137-148.

\bibitem{Pol} D. H. J. Polymath, \textit{A new proof of the density Hales--Jewett theorem}, preprint (2009), available at
\verb"http://arxiv.org/abs/0910.3926".

\bibitem{Ra} F. P. Ramsey, \textit{On a problem of formal logic}, Proc. London Math. Soc., 30 (1930), 264-286.

\bibitem{Rose} H. E. Rose, \textit{Subrecursion: functions and hierarchies}, Oxford Logic Guide 9, Oxford Univ. Press, Oxford, 1984.

\bibitem{Sh} S. Shelah, \textit{Primitive recursive bounds for van der Waerden numbers}, J. Amer. Math. Soc., 1 (1988), 683-697.

\bibitem{So} M. Soki\'{c}, \textit{Bounds on trees}, Discrete Math., 311 (2011), 398-407.

\bibitem{To} S. Todorcevic, \textit{Introduction to Ramsey Spaces}, Annals Math. Studies, No. 174, Princeton Univ. Press, 2010.

\end{thebibliography}
\end{document}